\documentclass[paper=a4, fontsize=11pt]{scrartcl}

\usepackage{amsmath,amsfonts,amsthm} 
\usepackage{bm} 
\usepackage{dsfont}
\usepackage[pdftex]{graphicx}	
\usepackage{enumitem}
\usepackage[numbers]{natbib}

\usepackage{url}
\usepackage{colortbl}
\usepackage{booktabs}
\usepackage{tabularx} 
\usepackage{float} 
\usepackage{multirow}%
\usepackage{fullpage}
\usepackage{natbib}
\usepackage{subcaption}  
\usepackage{graphicx}
\usepackage{tabularx} 
\usepackage{float} 
\usepackage{caption}

\usepackage{algorithm}
\usepackage{algpseudocode}
\usepackage{adjustbox}
\usepackage{xr-hyper} 

\usepackage[colorlinks=true,citecolor=blue,urlcolor=blue]{hyperref}
\usepackage{natbib}

\usepackage{graphicx}
\usepackage[english]{babel}
\usepackage{enumitem}
\usepackage{dsfont}
\usepackage{amsmath,amssymb}

\newcommand{\diam}{\operatorname{diam}}

\newcommand{\pr}{\mathbb{P}}

\newcommand{\argmin}{\operatornamewithlimits{\arg\min}}
\newcommand{\ind}{\mathds{1}}

\newcommand{\PP}{\mathbb{P}}
\newcommand{\RR}{\mathbb{R}}

\newcommand{\EE}{\mathbb{E}}

\newtheorem{theorem}{Theorem}
\newtheorem{definition}[theorem]{Definition}

\newtheorem{corollary}[theorem]{Corollary}
\newtheorem{proposition}[theorem]{Proposition}
\newtheorem{lemma}[theorem]{Lemma}

\usepackage[T1]{fontenc}

\usepackage[english]{babel}															
\usepackage{amsmath,amsfonts,amsthm} 
\usepackage{bm} 
\usepackage{dsfont}
\usepackage[pdftex]{graphicx}	
\usepackage{enumitem}

\usepackage{url}
\usepackage{colortbl}
\usepackage{booktabs}
\usepackage{tabularx} 
\usepackage{float} 
\usepackage{multirow}%
\usepackage{fullpage}
\usepackage{natbib}
\usepackage{graphicx}
\usepackage{tabularx} 
\usepackage{float} 
\usepackage{caption}

\usepackage{xr-hyper} 

\usepackage[colorlinks=true,citecolor=blue,urlcolor=blue]{hyperref}
\usepackage{natbib}

\usepackage{chngcntr}
\counterwithout{table}{section}
\counterwithout{figure}{section}

\newcommand{\horrule}[1]{\rule{\linewidth}{#1}} 	

\title{
		\vspace{-1in} 	
		\usefont{OT1}{bch}{b}{n}
		\horrule{0.5pt} \\[0.4cm]
		\Large A theory of shape regularity for local regression maps
		\\
		\horrule{2pt} \\[0.1cm]
}

\date{\large \today}

\author{ \textbf{Jérémy Bettinger, François Portier, Adrien Saumard} \\
\large \href{mailto:jeremy.bettinger@ensai.fr}{jeremy.bettinger@ensai.fr} ; \href{mailto:francois.portier@ensai.fr}{francois.portier@ensai.fr} ; \href{mailto:adrien.saumard@ensai.fr}{adrien.saumard@ensai.fr} \\
\large Department of Statistics, \\
\large University of Rennes, ENSAI, CNRS, CREST-UMR 9194, F-35000 Rennes, France }

\begin{document}

\maketitle

%

%

\begin{abstract}
 We introduce the concept of shape-regular regression maps as a framework to derive optimal rates of convergence for various non-parametric local regression estimators. Using Vapnik-Chervonenkis theory, we establish upper and lower bounds on the pointwise and the sup-norm estimation error,  even when the localization procedure depends on the full data sample, and under mild conditions on the regression model. Our results demonstrate that the shape regularity of regression maps is not only sufficient but also necessary to achieve an optimal rate of convergence for Lipschitz regression functions. To illustrate the theory, we establish new concentration bounds for many popular local regression methods such as nearest neighbors algorithm, CART-like regression trees and several purely random trees including  Mondrian trees.
\end{abstract}

\section{Introduction}\label{s1}
Consider the standard regression problem where the goal is to estimate the regression function of a random variable $Y\in \mathbb R$ given the covariates vector $X\in \mathbb R^d$, defined as $ g(x) := \mathbb E [ Y| X= x]$, $x\in \mathbb R ^d$. One leading approach, called \textit{local regression} or \textit{local averaging}, consists in averaging the observed response variables, restricted to covariates that lie in a small region of the domain $\mathbb R ^d$. Local regression methods include kernel smoothing regression \cite{nadaraya1964estimating}, nearest neighbors algorithm \cite{fix1989discriminatory,cover1968estimation} and regression trees or more generally partitioning regression estimator \cite{breiman1984classification,nobel}. We refer to the books
\cite{devroye96probabilistic,gyorfi2006distribution} for a complete view on local regression methods and to \cite{biau2015lectures} for a precise account on the theory of nearest neighbors algorithm.

Concerning the estimation problem, when the error is measured in terms of the mean squared error ($L_2$-error), optimal rates of convergence are known \cite{stone1982optimal} and they depend on the smoothness of the regression function $g$. Achieving or not these convergence rates often serves as a theoretical baseline to evaluate the accuracy of local regression methods. For instance, a Lipschitz function $g$ can only be approximated at the rate $ n ^{- 1/ (d+2)  }$ in general, when $n $ independent  observations are given. Many of the above estimators are known to achieve optimal convergence rates. The nearest neighbors, the Nadaraya Watson and the fixed partitioning (histogram) regression estimators are all optimal for Lipschitz functions (as well as for twice differentiable functions for the first two listed methods), as explained in \cite{biau2015lectures},  \cite{tsy_08} and Chapter 4 of \cite{gyorfi2006distribution}, respectively. Furthermore, the Nadaraya-Watson \cite{einmahl2000,gine+g:02} and the nearest neighbors \cite{jiang2019non,portier2021nearest} estimators are both known to achieve a rate of convergence in sup-norm that is of the same order as the $L_2$-rate (up to a logartimic term). 

Other local regression estimators that are based on purely random trees are of interest \cite{biau2016random}, despite the independence of the leaves with respect to the original data, because of their ability to explain certain patterns in the success or failure of different tree constructions and to illustrate the success of random forest regression. Among purely random trees, Mondrian trees, as introduced in \cite{lakshminarayanan2014mondrian}, indeed achieve the optimal rate of convergence \cite{mondrian} for Lischitz functions while, in contrast, centered trees fail to reach it \cite{biau2012analysis,klusowski2021sharp}. We also highlight that partitioning the space with Voronoï cells allows to attain the optimal rate of convergence \citep{gyorfi2021universal} (see also \cite{hanneke2021universal} and the reference therein for more results).

Despite the many existing results available for the Nadaraya-Watson and the nearest neighbors regression estimators, and also fixed or purely random  partitioning regression rules, only few is known about local regression based on data-dependent partitions such as the well-known CART regression tree \cite{breiman1984classification}. Such an algorithm is indeed much harder to precisely analyze mathematically.
First results on data dependent partitions can be found in \cite{stone1977consistent}, but they are restricted to cases where the partition depends only on the covariates, as in nearest neighbors regression or for statistically equivalent blocks \cite{anderson}. 
More advanced results, that are valid for general data dependent partitioning estimators, are obtained in \cite{gordon1980consistent,breiman1984classification,nobel},  where conditions are given to ensure almost sure $L_2$-consistency. The typical assumptions that are required in the previous works include (i) large enough points in each partition element and (ii) small diameter, while having (iii) a reduced complexity on the partition elements. Note also that Theorem 1 in \cite{scornet2015consistency} can be applied to CART regression algorithm and gives sufficient conditions for the $L_2$-consistency.  

Beyond consistency, few is known about the convergence rates of data-dependent, CART-like regression tree estimators. Recent studies  \cite{chi2022asymptotic, mazumder2024convergence} have obtained convergence rates for the $L_2$-error under the so-called \textit{sufficient impurity decrease} (SID) condition, that is directly linked to the behavior of the precise splitting rule of CART in the regression context. The rate of convergence obtained depends on a parameter - denoted $\lambda$ in \cite{ mazumder2024convergence} - quantifying the strength of the SID condition, and it is not \textit{a priori} easy to discuss the rate optimality. Nonetheless, it is shown in  \cite{ mazumder2024convergence} that for a univariate linear regression function, the rate obtained through the SID condition is actually optimal. A specific class of additive regression functions achieving a particular smoothness assumption called the "locally reverse Poincaré inequality" is provided in  \cite{ mazumder2024convergence}, satisfying the SID condition. In another direction, the recent negative results in \cite{cattaneo2022pointwise} show that CART regression can be sub-optimal, and even inconsistent, for the pointwise - and also uniform - estimation error. Such phenomenon does not occur when focusing on the $L_2$-error, but as highlighted in \cite{cattaneo2022pointwise}, pointwise convergence of decision trees is also essential for reliability of the methodologies developed in some causal inference and multi-step semi-parametric settings for instance.

In this work, we develop a theory for obtaining pointwise and uniform rates of convergence for a large class of local regression estimators, that includes previously mentioned partitioning estimators. More precisely, in a random design regression with heteroscedastic sub-Gaussian noise framework, the theory allows the localization method to be general as it may depend on a different source of randomness (as for the purely random tree) or on the covariates sample (as for nearest neighbors) and even on the full regression sample (as in CART). Instead of studying the integrated $L_2$-error, our approach deals with the pointwise and uniform estimation errors recently put forward in the literature \cite{cattaneo2022pointwise}, for which we obtain a sharp probability upper bound (Theorem \ref{th:general}). To prove such a result, we proceed with a decomposition of the pointwise estimation error into the sum of a variance term 
(scaling as the inverse of the square root of the number of covariates in the partition elements) 
and a bias term (scaling as the diameter of the partition elements). We point out that the major advantage of focusing on the pointwise error, compared to the $L_2$-error, is that it allows the control of the variance and bias terms though the use of the Vapnik dimension of a class containing the \textit{elements} of the random partition, instead of having to control the combinatorial size of the class of the \textit{entire} partitions themselves as in \cite{lugosinobel}.

Next, we introduce the notion of \textit{shape regularity} by imposing a simple relationship between the Lebesgue volume and the diameter of the localizing set. The major interest of this simple geometric property is that it turns out to be necessary (Proposition \ref{contre_ex}) and sufficient (Theorem \ref{main_result_localizing_map}) for obtaining optimal rates in pointwise and uniform estimation error. We then discuss how several tree constructions, including purely random trees such as uniform and Mondrian trees, satisfy - or not - the shape regularity condition, allowing to obtain  - or not - optimal rates of convergence. In addition,  the shape regularity allows to recover and slightly extend some results pertaining to the nearest neighbors literature \cite{jiang2019non,portier2021nearest}. 

Finally, we obtain a deviation inequality on the uniform estimation error of CART-like regression trees, grown by ensuring a minimal number of covariates in the tree leaves and by following a simple rule which maintains the shape regularity of the resulting localizing sets. In the case of partitions made of hyper-rectangles, such as for CART-like algorithms, the shape-regularity condition reduces to a control of the largest side length of the localizing set by its smallest side length. Recent results obtained in \cite{cattaneo2022pointwise} indeed tend to indicate that such rules addition is unavoidable to ensure good pointwise convergence rates for CART regression. 

It is worth noting that our approach substantially differs from the use of the SID condition described earlier. The latter indeed ensures convergence rates for the $L_2$-error and is highly linked to the precise cost in the splitting rule of CART, defined through the so-called impurity gain. Moreover, the SID condition is expressed through the behavior of the unknown regression function and covariates distribution, and cannot hold for any regression function. In contrast, our shape-regularity condition does not depend on the regression function $g$, neither on the covariates distribution, and only imposes restriction that may be effective with any cost function involved in the splitting rule. This makes our shape regularity condition easy to guarantee in practice as illustrated in Algorithm 1 (see Section \ref{s51}), where a general cost function is used to build the tree.

The outline is as follows. We state in Section \ref{s2} some necessary background and formulate the setting of local regression map estimators. Section \ref{s30} then gives a first deviation inequality for local regression map estimators. We introduce in Section \ref{s3} the shape regularity conditions and their properties. Section \ref{s5} is dedicated to pointwise and uniform convergence bounds for  data-dependent regression maps, namely nearest neighbors and CART-like trees. Finally, Section \ref{sec_PRT}  includes new positive and negative results about some classical purely random trees. All the mathematical proofs are given in the Appendix.

\section{Mathematical background}\label{s2}

\subsection{Regression set-up}\label{s21}

Let $(X,Y)$ be a random vector with probability distribution $\mathbb P$ on $\mathbb R^d \times \mathbb R$, where $d\geq 1$ is the dimension of covariates vector $X$. Consider the standard regression framework where the random variable $X\in S_X \subset\mathbb R^d$ is called the covariates and $Y \in \mathbb R$ is the output variable.  We aim at estimating the conditional expectation $x\mapsto g(x) = \mathbb E [ Y|X= x]$, $x\in S_X$. The quality of the estimation of the function $g$ by an estimator $\hat g$ will be assessed with the help of the uniform norm defined as $\sup_{x\in S_X} | \hat g(x) - g(x) | $. 
For a fixed $x\in S_X$, we also address the estimation error of the value $g(x)$ through the analysis of the deviations of the quantity $| \hat g(x) - g(x) | $. 

The following assumption on $\mathbb P$ will be key in this work and, roughly speaking, amounts to assume that the noise $\epsilon = Y  - g(X)$ in the regression model is light tailed.

\begin{enumerate}[label=(E), wide=0.5em,  leftmargin=*]
\item \label{cond:epsilon} 
The random variable $\epsilon$ is sub-Gaussian conditionally on $X$ with parameter $\sigma^2$. That is, $\mathbb E [\epsilon |X ]= 0$ and for all $\lambda \in \mathbb R$, $$ \mathbb E [ \exp( \lambda\epsilon  ) |X ] \leq \exp\left( \frac{ \lambda ^2}{   2\sigma^2 }\right).$$
\end{enumerate}
Note that under assumption \ref{cond:epsilon}, the noise term $\epsilon$ is squared integrable and it is allowed to depend on the covariates $X$.  In particular, the noise is \textit{heteroscedastic}, with a uniform upper bound on its conditional variance: $\mathbb{E}[\epsilon^2\vert X] \leq \sigma^2$ $a.s.$. A more restrictive assumption is when $\epsilon$ is independent from $X $ and sub-Gaussian with parameter $\sigma^2$.

Associated to the regression model, we consider a sample $\mathcal{D}_n=\left\{(X_i,Y_i) \,  : \, i=1,\ldots,n\right\}$ of independent and identically distributed pairs of random variables with distribution $\mathbb  P^{\otimes n}$. Let also $\varepsilon_ i = Y_i - g(X_i) $ for each $i=1,\ldots, n$.

\subsection{Local regression maps}\label{s22}

We consider general local regression estimators using the concept of local maps so as to include regression trees and partitioning estimators but also the nearest neighbors regression rule. Let $\mathcal B(S_X)$ denote the Borel $\sigma$-algebra on $S_X$.

\begin{definition}
    A local map for a variable $X$ is a mapping $ \mathcal V : S_X \to \mathcal B(S_X) $ such that for all $x\in S_X$, $x\in \mathcal V (x)$. 
\end{definition}

We emphasize here that this work deals with continuous covariates and therefore all local maps will have their images to sets with positive Lebesgue measure. Also we stress out that similar maps were introduced in \cite{nobel}, where they are however restricted to partition based estimator.
For any local map $\mathcal V$, the associated regression estimator is given by
$$\forall x \in S_X, \quad \hat g_{\mathcal V}(x) = \frac{\sum_{i=1} ^ n Y_i\mathds 1 _{ \mathcal V(x) }(X_i) }{\sum_{i=1} ^ n \mathds 1 _{ \mathcal V(x) }(X_i)  },$$ with the convention $0/0 = 0.$
Local maps  $\mathcal V $ depending on the sample $(X_1,Y_1),\ldots , (X_n, Y_n)$ are of particular interest. This is indeed the case for certain adaptive tree constructions as well as for nearest neighbors regression.  

The local regression map framework is particularly interesting because it includes a variety of different methods, e.g., fixed partitioning,  purely random trees, nearest neighbors, and CART-like constructions, and each method induces a particular dependence structure when creating the partition.

\paragraph{Example 1 (fixed hyper-rectangles partition).} The most simple case for the dependence structure of the local map is when the partition is fixed, not random.  Suppose $S_X= (0,1]^d$. For each coordinate $k =1,\ldots, d$, consider the collection 
$ 0=  u_0^{(k)} < u_1^{(k)}< \ldots < u_{N_k}^{(k)}  =   1 $. 
This allows to introduce a partition of $S_X$ made of $ \prod_{k=1} ^d N_k $ elements defined as $ V_{i_1,\ldots, i_d} =  \prod_{ k = 1 }^d (u_{i_k}^{(k)}, u_{i_k + 1 }^{(k)} ]$ for each d-uplet $(i_1,\ldots, i_d)$ satisfying $i_\ell\in \left\{0,\ldots,N_\ell-1 \right\}$ for $\ell\in \left\{1,\ldots,d \right\}$. Note that each $V_{i_1,\ldots, i_d}$ has a positive Lebesgue measure $ \prod_{ k = 1 }^d (u_{i_k + 1 }^{(k)} - u_{i_k}^{(k)} )$.

\paragraph{Example 2 (purely random trees).}
In contrast with Example 1, \textit{purely random tree} construction, as described in \cite{arlot2014analysis} and initially introduced in \cite{breiman2000some},  follows from using some randomness that is independent from the observed sample. It includes centered (resp. uniform) trees, for which the split direction is uniformly distributed along the space coordinates and the split location of the selected side is at the center (resp. uniformly distributed).  It also includes Mondrian trees \citep{lakshminarayanan2014mondrian} where the split direction is selected at random depending on the shape - i.e. side lengths - of the leaf.  

\paragraph{Example 3 (nearest neighbors regression).} Nearest neighbors algorithm induces a Voronoï-like partition which dependence structure is different from the one of purely random trees, since the resulting partition depends here on the data through the location of the covariates in the space. The $k$-nearest neighbors ($k$-NN) estimator (see \cite{biau2015lectures} for a recent textbook) is defined, for each $x\in S_X$, as the average responses among the $k$-nearest neighbors to point $x$. As such, we have
$$ \hat g_{NN}(x) =  \frac{1}{k }  \sum_{i=1}^n  Y_i\ind_{B(x, \hat \tau_k(x))} (X_i ) ,$$
where $\hat \tau_k(x)$ is the so-called $k$-NN radius defined as the smallest radius $\tau>0 $ such that $\sum_{i=1}^n  \ind_{B(x,  \tau ) } (X_i ) \geq k$.
Note that here the local map is $\mathcal V (x)  = B(x, \hat \tau_k(x))$ and therefore  depends on $X_1,\ldots, X_n$.

\paragraph{Example 4 (CART-like trees).} 
Regression trees are a class of partition based estimators where the partition is recursively built, and made of hyper-rectangles. Therefore, they are part of the local map framework, just as example 1 and 2 above. Usual regression trees are grown sequentially by splitting stage-wise each (adult) leaf into two (children) leafs. In most cases, as in CART regression \citep{breiman1984classification}, each cell division results from splitting along one single variable according to a data-based criterion. This precise step is crucial as it allows to adapt the partition to the prediction problem. For instance, if one variable is not significant then it must be better not to split with respect to it. This enables to obtain a flexible regression estimator which behaves well in many problems even when the dimension $d$ is rather large. The fact that the resulting partition depends on the full data (including the response) is however problematic for the theory since in this case, the local averaging estimator is not a sum over independent random variables, thus prohibiting a direct application of concentration inequalities for sums of independent observations. Finally, it is worth mentioning that CART regression trees are the ones that are usually combined in the standard random forest regression algorithm as introduced in \cite{breiman2001random}.
 

\section{A deviation bound for local map estimators}\label{s30}

The section is divided into two parts. We first give some preliminary concentration bounds that are free from any restriction on the probability distribution of the covariates. Then we use them in the regression framework in order to have some concentration bounds on the estimation error.

\subsection{A deviation bound for the variance term}\label{3.1}

The \textit{shattering coefficient}, as introduced in Vapnik's seminal work \cite{vapnik2015uniform} and  detailed for instance in \cite{wellner1996},
is key to obtain upper bounds on certain empirical sums indexed by sets or functions. Let $\mathcal A$ be a collection of subsets of a set $ S$. Given an arbitrary collection $z =( z_1,\ldots, z_n)$ of distinct points in $  S$, consider the collection of $\mathbb R^n$-points
$ \ind _ {\mathcal A} (z)$ defined as $  \{ (\ind _ A (z_1) \ldots, \ind_A (z_n )) : A\in \mathcal A \}\subset \{0,1\}^n $. 
We have that $|\ind _ {\mathcal A} (z)  | \leq 2^n$ and when $   |\ind _ {\mathcal A} (z) |  = 2^n$ we say that $z$ is shattered by $\mathcal F$. An important quantity is then
 $$\mathbb S_\mathcal A(n) : =  \sup_{z\in \mathbb R^n} | \ind _ {\mathcal A} (z) | $$
which is called the shattering coefficient. 

We now provide a VC-type inequality tailored to local regression maps. This will be key to the analysis of the variance term of local map estimators.

\begin{theorem}\label{1}
Let $n\geq 1$ and $\delta \in (0,1) $. Suppose that \ref{cond:epsilon} is fulfilled and that $\{\mathcal V (x) \, :\, x\in \mathbb R^d\} \subset \mathcal A$, a deterministic collection of sets in $\mathbb R^d$. The following inequality holds with probability at least $1 - \delta$,
$$\sup_{x\in \mathbb R^d} \dfrac{\sum_{i=1}^n \varepsilon_i \mathds 1 _{ \mathcal V(x) }(X_i)}{\sqrt{\sum_{j=1}^n \mathds 1 _{ \mathcal V(x) }(X_j)}} \leq \sqrt{2 \sigma^2 \log\left( \frac{\mathbb S_ {  \mathcal A   } (n)}{\delta} \right)}.$$
\end{theorem}
Note that in Theorem \ref{1} above, only an upper bound is given but a lower bound is also valid, since the same holds true when each $\varepsilon _ i $ are replaced by $ -\varepsilon_i$. Moreover, combining such inequalities through a union bound gives a result for the supremum of the absolute value.

\subsection{A pointwise error bound}\label{s23}

We now state a general deviation bound on the uniform error of local regression map estimators with finite Vapnik-Chervonenkis (VC) dimension. The VC dimension is defined as 
\begin{align*}
  vc(\mathcal A)  &= \max \{ n\geq 1 \, :\, \mathbb S_\mathcal A(n) 
 = 2^n  \}. 
\end{align*}
As a consequence, the fact that all given $z_1,\cdots , z_{v+1} $ points cannot be shattered is equivalent to the fact that the VC dimension is smaller than $v$. The reason why the VC dimension is appropriate to control the complexity of classes of sets is perhaps explained by the Sauer's lemma (see \cite{lugosi2002pattern} for a proof) which states that
$\mathbb{S}_{\mathcal{A}}(n)  \leq  \sum_{i=0} ^{vc(\mathcal{A})} \binom{n}{i}.$ 
  An interesting consequence of Sauer's lemma is that 
$\mathbb S_\mathcal A(n) \leq (n+1)^{vc(\mathcal A) }.$

As established in \cite{wenocur1981some}, previous examples include the class of cells $(-\infty , t]\subset \mathbb R^d$, having VC dimension equal to $d$, or the class $( s , t]$, $s,t\in \RR^d$, of VC dimension equal to $2d$. In addition, the class of balls in $\mathbb R^d$ has dimension equal to $d+1$. 
\begin{definition}
 A local map $\mathcal V$ is said to be VC whenever $\{\mathcal V (x) \, :\, x\in S_X\} \subset \mathcal A$, a fixed VC collection of sets in $\mathbb R^d$.
\end{definition}
Let us further define some quantities that will be instrumental in our analysis. For any set $V$, its diameter is given by the formula
$$ \diam (V)  = \sup_{(x,y)\in V\times V} \|x-y\|_2,$$
where $\|x\|_2^2 = \sum_{k=1} ^d x_k^2$. A real function $g$ on $S_X$ is called $L$-Lipschitz as soon as $ | g(x) - g(y) |\leq L \|x-y\|_2$ for all $(x,y)\in S_X^2$.
 Define also the local Lipschitz constant $ L(V)$ of $g$ on $V\subset S_X$ as the smallest constant $L>0$ such that, for all $(x ,y)$ in $V^2$,
$$     |g(x) - g(y) | \leq L  \|  x-y\|_2.$$
For a $L$-Lipschitz function, it holds $L(V)\leq L$ for any set $V\subset S_X$.

The next probability error bound is valid for local map estimators, with a general VC local map, that may for instance depend on the sample.

\begin{theorem}\label{th:general}
Let $n\geq 1$ and $\delta \in (0,1/2) $. Under \ref{cond:epsilon}, suppose that $g$ is Lipschitz on $S_X$ and that the local map is VC with dimension $v$. We have, with probability at least $1 - 2 \delta $, for all $x\in S_X$,
        \begin{align*}
       |  \hat g_{\mathcal V}(x) - g(x)| \leq \sqrt{\frac{ 2 \sigma^2 \log\left( \frac{ (n+1) ^v }{\delta} \right)}{ n \PP_n(\mathcal V (x))  } } + L(\mathcal V(x)  ) \diam  (\mathcal V(x) )  .
    \end{align*}
\end{theorem}

An alternative approach proposed in \cite{lugosinobel,nobel} as well as in \cite{devroye96probabilistic}, see Theorem 21.2 therein, follows from a uniform control on all resulting partitions, implying consistency results for sums over all partition elements. In Theorem \ref{th:general}, our approach is substantially different, since by considering the pointwise or sup-norm error,  the complexity term comes from the elements of the partition only. In addition, Theorem \ref{th:general} above might be compared with Theorem 6.1 in \cite{devroye96probabilistic}, which is suitable to either non-random or purely random (i.e., independent from the sample) data partitioning (\cite{JMLR:v9:biau08a}). While Theorem \ref{th:general} is valid for data dependent partitions, we recover almost sure consistency by imposing two conditions that are similar to those required in Theorem 6.1 of \cite{devroye96probabilistic}, namely $\diam  (\mathcal V(x) )  \to 0 $ and $ n\PP_n( \mathcal V(x) ) / \log(n )  \to 0$. Depending on whether the previous conditions hold uniformly in $x$ or for a given $x$, the consistency, uniform or pointwise, of the local map regression estimator can thus be obtained.

\section{Shape regularity}\label{s3}

We describe here the minimal mass assumption, which deals with the distribution of the covariates $\PP$. We then introduce the concept of shape regularity for local maps.

\subsection{Minimal mass assumption}\label{4.1}

The next minimal mass assumption allows to obtain an estimate for $\PP_n(\mathcal V(x)) $, which appears in the upper bound stated in Theorem \ref{th:general}.

\begin{enumerate}[label=(X), wide=0.5em,  leftmargin=*]
\item  \label{cond:density_X} For a local map $\mathcal V$ on $S_X$, there exist a constant $\kappa >0 $ and a density function $f_X$ such that, almost surely, for all $x\in S_X$, 
$$ \PP(\mathcal V(x) )  \geq  \kappa f_X(x) \lambda (\mathcal V(x)  )  ,$$
where $\lambda$ stands for Lebesgue measure on $\mathbb R^d$.
\newcounter{nameOfYourChoice}
\setcounter{nameOfYourChoice}{\value{enumi}}
\end{enumerate}
The minimal mass assumption is quite general as it allows to include $k$-NN regression estimators as well as partitioning estimators on $[0,1]^d$. In the first case, $\mathcal V (x)$ is a (small enough) ball with positive (random) radius. As shown in \cite{jiang2019non}, such assumption is satisifed by bounded from below density on smooth sets $S_X$, as well as Gaussian or Laplace variables on $\mathbb R^d$ (\cite{gadat2016classification}). Some further details will be given in Section \ref{s61}. In the second example $\mathcal V (x) $ is an hyper-rectangle included in $S_X= [0,1]^d$ and therefore lower bounded densities on $S_X$ easily satisfy \ref{cond:density_X}.   

The following definition is now required to ensure that enough points are lying within each element of the local map. 

\begin{definition}
A VC local map  $x\mapsto  \mathcal V(x)$ with dimension $v>0$ is called $(\delta, n)$-large whenever, for all $x\in S_X$, almost surely,
    \begin{align*}
  n \max (\PP_n(\mathcal V(x))  ,     \PP(\mathcal V(x)) )  \geq 8 \log\left(\dfrac{4 (2n+1) ^v }{\delta} \right).
\end{align*}

\end{definition}

Note in particular that the latter inequality is easy to check in practice as it suffices to make sure that enough points are in each element of the local map.

\begin{theorem}
    \label{th2:general}
   Let $n\geq 1$ and $\delta \in (0,1/3) $.  Under \ref{cond:epsilon} and \ref{cond:density_X}, suppose that $g$ is $L$-Lipschitz on $S_X$, that the local map is VC with dimension $v$ and is $(\delta, n)$-large, then we have with probability at least $1-3\delta$, for all $x\in S_X$,
\begin{align*}
        | \hat g_{\mathcal V}(x) - g(x)| \leq  \sqrt{\frac{ 3 \sigma^2 \log\left( \frac{(n+1)^v}{\delta} \right) }{n  \kappa f_X(x) \lambda ( \mathcal V (x) )   }} + L(\mathcal V(x) ) \diam (\mathcal V(x) ).
    \end{align*}
\end{theorem}

The previous result is different from the one of Theorem \ref{th:general} as the bound no longer depends on the number of points in the associated sets but instead on the volume. Together with the diameter, these two quantities will be appear in the definition of the $\gamma $-shape regularity, so as to minimize the latter upper bound and therefore, to attain optimal rates of convergence for the underlying regression problem.


\subsection{Shape-regular sets}\label{4.2}

A key concept, which will help us to characterize the rates of convergence of the local map regression estimators, is now introduced. As established in Theorem \ref{th2:general}, under the minimal mass assumption, the quantity $ |\hat g_{\mathcal V}(x) - g (x) |$ is bounded by
 $  \sqrt { 1 / ( n \lambda (\mathcal V (x)) )    }   +   \diam (\mathcal V (x))  $, up to constants and log terms. Theorem \ref{th2:general} allows to understand that a trade-of between the volume and the diameter must be achieved in order to reach optimal rates.
In this regard, first note that the volume cannot be larger than the diameter to the power $d$, as  $ \lambda (\mathcal V(x) )  \leq \diam ( \mathcal V(x) )^d $. 
Then optimizing the previous bound with respect  to $ \lambda (\mathcal V(x) )$ and $\diam ( \mathcal V(x) )$  under constraint that $\lambda (\mathcal V(x) )  \leq \diam ( \mathcal V(x) )^d $  leads to $\lambda (\mathcal V(x) ) = \diam ( \mathcal V(x) )^d = n^{-d/(d+2)}$ which leads to optimal rate. On the contrary, if $ \diam ( \mathcal V(x) )^d = \gamma_n \lambda (\mathcal V(x) )   $ with $\gamma_n \to \infty$, then the rate of convergence is suboptimal. This reasoning motivates the introduction of the following notion of shape-regularity.

\begin{definition}\label{def:gamma_regular}
For $\gamma>0 $, a set $ V $ is called $\gamma$-shape-regular ($\gamma$-SR) if
$\diam(V)^d \leq \gamma \lambda (V)$.
\end{definition}
\noindent The previous condition can be interpreted as a volume condition: the volume of $V$ should be of the same order as the volume of the smallest ball containing $V$. Roughly speaking, the shape of $V$ is not that different from that of a ball. Moreover, it does not depend on the covariates density, making it easy to check in practice. 
 
We provide now an alternative to Definition \ref{def:gamma_regular}, specifically designed for local maps valued in the set of hyper-rectangles. 
For any hyper-rectangle $A\subset  S_X$, let $h_-(A)$ and $h_+(A)$ denote the smallest and largest side length, respectively.
\begin{definition}\label{def:beta_regular}
For $\beta>0 $, a hyper-rectangle $ A$ is called $\beta$-shape-regular ($\beta$-SR) if 
$ h_+ (A)  \leq \beta h_-(A)$.
\end{definition}
\noindent It is easily seen that when a set $V$ is an hyper-rectangle, the $\gamma$-SR property is related to $\beta$-SR. This is the subject of the following proposition. 
\begin{proposition}\label{link_beta_gamma}
     A $\beta$-SR hyper-rectangle  is $\gamma$-SR with $\gamma = \beta^d \, d^{d/2}$. Conversely, a $\gamma$-SR hyper-rectangle  is $\beta$-SR with $\beta= \gamma$.
\end{proposition}
\noindent The two definitions of shape regularity, $\gamma$ and $\beta$, are therefore equivalent in the case of hyper-rectangles. More precisely, the first implication from above will be of particular interest for us, as it will allow to show that some regression trees are $\gamma$-shape-regular. In practice, one way to obtain a $\beta$-SR (and therefore $\gamma$-SR) tree is to allow only for $\beta$-SR splits when growing the tree, i.e., valid splits in light of Definition \ref{def:beta_regular}. This is easily imposed as it only requires to restrict the  optimization domain when finding the optimal split. We further develop this aspect in Section \ref{s51} below.

Note that, in dimension $d=1$, trees are necessarily shape-regular for $\gamma=\beta=1$ as $ h_- = h_+$. From this perspective, dimension $1$ plays a special role and might exhibit convergence properties that would not generalize to larger dimensions.

Let us now formalize a bit more on the idea that a non-shape-regular set would lead to a suboptimal convergence rate, at least for some regression functions that are sufficiently varying. Consider indeed the function $g:x\mapsto \sum_{k=1} ^d x_k $ defined on $[0,1]^d$. Set $d\geq 1$ and assume that $X \sim \mathcal U [0,1]^d$. Since $g$ is Lipschitz - note that each partial derivative of $g$ is actually \textit{equal} to one pointwise - optimal rates are of order $ n^{ - 1/(d+2)}$. Consider estimating $g$ at $0$ using a rectangular cell such that $\diam(\mathcal V)^d/ \lambda (\mathcal V) \geq \bar {\gamma}$ where $\bar {\gamma} >0$. Next we show that, under standard conditions, the optimal rate cannot be achieved when $\bar \gamma $ grows with $n$. This is important as it means that the optimal rate cannot be attained except when $\bar{\gamma}$ is bounded, meaning that trees need to be shape-regular for being optimal.

\begin{proposition}\label{contre_ex}
Let $n\geq 1$ and $d\geq 1$. Suppose that $X \sim \mathcal U [0,1]^d$ and that \ref{cond:epsilon} is fulfilled with $ g(x) = \sum_{k=1} ^d x_k$. Consider a localizing map $\mathcal V$ such that $ \mathcal V(0) = \prod_k[0,h_k]$. Let $\bar{\gamma}$ be such that $\diam(\mathcal V)^d/ \lambda (\mathcal V) \geq \bar{\gamma}$. Whenever $2^{d+4} \log(2) \leq  n \prod_{k=1}^d h_k$, there exists a constant $C_d > 0$ depending only on $d$ such that $$   \EE [ ( \hat g_{\mathcal V} (0 ) -  g(0 ))^2 ]^{1/2}  \geq  C_d \left( \dfrac{\bar{\gamma} \sigma^2}{n} \right)^{1/(d+2)}. $$
\end{proposition}






More generally, the latter result still holds $g(x) \geq \sum_{k=1}^d x_k - g(0)$ on $\mathcal V (0)$. An example of such function is for instance $g$ differentiable, $\nabla g (0)= (1,...,1)^T$ and $g$ convex. But many non-convex functions satisfy this condition of course. Note also that the previous result can be extended to covariates $X$ having a density uniformly bounded from above and from below. Note finally, that the lower bound provided in Proposition \ref{contre_ex} is also valid when for the $L_2(P^{\otimes n})$-norm of the sup-norm estimation error $\Vert g_{\mathcal V} - g\Vert_\infty \,$, since the latter quantity is bounded from below by the pointwise estimation error.

\subsection{Shape regularity of local maps}\label{4.3}

Let us now introduce the following definition which requires that all elements of the localizing map are $\gamma $-SR. 

\begin{definition}
A localization map  $x\mapsto  \mathcal V (x)$ is $\gamma$-SR  if all elements in $\{\mathcal V (x)\,:\, x\in S_X\}$ are $\gamma$-SR. 
\end{definition}

We stress that for most trees, the randomness of the construction will require to study the stochastic variability of $\gamma$ (see Section \ref{sec_PRT} where uniform, centered and Mondrian tree are considered). To validate the $\gamma$-SR condition, we now provide an optimal error rate for such $\gamma$-SR local maps when choosing a good value for the volume. In the next statement, we use the notation
$ f \lesssim g $ when there exists a universal constant $ a >0$ such that 
$   f  \leq a  g .$
We write $ f \asymp g $ whenever $f\lesssim g $ and $g\lesssim f$.

\begin{theorem}\label{main_result_localizing_map}   
Under the assumptions of Theorem \ref{th2:general}, if the local map is $\gamma$-SR and if for all $x\in S_X$,
$\lambda(\mathcal V(x) ) \asymp ({\log ( {(n+1)^v} / { \delta} ) } / { n} )  ^{d/(d+2)}$, we have, with probability at least $1 - 3\delta $, for all $x\in S_X$,
    \begin{align*}
        & | \hat g_{\mathcal V}(x) - g(x)| \lesssim c \left(\frac{\log\left( \frac{(n+1)^v}{ \delta} \right) }{ n} \right) ^{1/(d+2)}
    \end{align*}
    where $c = \sqrt{{3 \, \sigma^2}/({\kappa f_X(x)})} +  L(\mathcal V(x) )  \gamma^{1/d}.$ In addition, whenever $S_X$ is bounded and $ f_X(x) \geq b >0 $ for all $x\in S_X$, we have, with probability at least $1 - 3\delta $,
 \begin{align*}
        & \sup_{x\in S_X} | \hat g_{\mathcal V}(x) - g(x)| \lesssim c \left(\frac{\log\left( \frac{(n+1)^v}{ \delta} \right) }{ n} \right) ^{1/(d+2)}
    \end{align*}
    where $c = \sqrt{{3 \, \sigma^2}/({\kappa b })} +  L  \gamma^{1/d}.$
    \end{theorem}


\section{Data-dependent local regression maps}\label{s5}

In this section, we show that shape regularity is useful to analyse local regression maps that are data-dependent. The first example is the nearest neighbors regression estimator and the second one is a modified version of the CART-like algorithm.

\subsection{Nearest neighbors regression}\label{s61}

Nearest neighbors regression estimators are local maps estimators for which
$\mathcal V (x) = B (x, \hat \tau _{n,k}(x) ) $
where $ \hat \tau _{n,k}(x)$ has been defined in Section \ref{s2}, Example 3. Using the standard index order to break possible ties, we have 
$ \PP_n (\mathcal V (x)  ) = k/n  $
and, by relying on Theorem \ref{th:general}, we obtain the following result.

\begin{theorem}\label{th:NN}
Let $\delta \in (0,1/3) $, $n\geq 1$, $d\geq 1$ and $k\geq  8    \log( 4 (2n+1) ^{(d+1) } / \delta ) $. Let $\mathcal V $ be obtained from nearest neighbors algorithm as detailed in Example 3. Suppose that \ref{cond:density_X} is  valid for all balls with radius smaller than $T_0>0$. Suppose that \ref{cond:epsilon} is fulfilled, and that $g$ is $L$-Lipschitz on $S_X$. We have with probability at least $1-3\delta$, for all $x\in S_X$ such that $2 k \leq   T_0^d n \kappa f_X(x)$,
\begin{align*}
    &|\hat g_{\mathcal V} (x) - g(x) | \leq  \sqrt{\frac{2\sigma^2\log ( {(n+1)^{d+1}/\delta })  }{ k }} + 2 \left(\frac{ 2 k   }{  n \kappa f_X(x)  }   \right)^{1/d} L(\mathcal{V}(x)).
\end{align*}
\end{theorem}

Note that the conditions on the value of $k$ are satisfied for $n$ sufficiently large   and $k\asymp n^a $, for any $a\in (0,1)$.   The assumption that \ref{cond:density_X} is  valid for all balls with radius smaller than $T_0>0$ is easy to check when the density is lower bounded by a constant $b>0$ on $S_X$ (in which case $S_X$ must be compact) and when $S_X$ satisfies $ \int _{S_X \cap B(x,\tau )} \geq \kappa_0 \tau ^d $ for all $\tau \leq T_0$, for some $\kappa_0 >0$ and $T_0>0$. This is done in \cite{jiang2019non,portier2021nearest} and extended to unbounded sets $S_X$ in \cite{gadat2016classification}. 
\color{black}

To our knowledge the above result is new among the nearest neighbors literature, in which uniform deviation inequalities are provided, to our best knowledge, for densities uniformly bounded away from $0$. Such results have been investigated recently in \cite{jiang2019non} and \cite{portier2021nearest} for compactly supported covariates. In contrast, the above upper bound is valid for all $x$ in any domain $S_X$, at the price of accounting for regions with low density value that may deteriorate the accuracy. We have the following corollary in which we consider an optimal choice for $k$ as well as a lower bounded assumption on the density.

\begin{corollary}\label{corknn}
Assuming that $ n $ is sufficiently large, then choosing the integer $k \asymp n^{2/(d+2)} \log((n+1)^{d+1}/\delta)^{d/(d+2)}$ in Theorem \ref{th:NN} yields the following inequality with probability at least $1-3\delta$, for all $x\in S_X$, 
   \begin{align*}
      |\hat g_{\mathcal V}(x) - g(x) |  \lesssim c \left( \dfrac{\log((n+1)^{d+1}/\delta)}{n}\right)^{1/(d+2)},
   \end{align*}
    where $c =   \sqrt{2\sigma^2} + 2 L(\mathcal{V}(x)) \left[{2}/({\kappa f_X(x)}) \right]^{1/d} .$  In addition, whenever $S_X$ is bounded and $ f_X(x) \geq b >0 $ for all $x\in S_X$, we have, when $n$ is sufficiently large, with probability at least $1 - 3\delta $,
 \begin{align*}
        & \sup_{x\in S_X} | \hat g_{\mathcal V}(x) - g(x)| \lesssim c \left( \dfrac{\log((n+1)^{d+1}/\delta)}{n}\right)^{1/(d+2)},
    \end{align*}
    where $c =   \sqrt{2\sigma^2} + 2 L \left[{2}/({\kappa b}) \right]^{1/d} .$  
\end{corollary}

We note that the convergence rate is the same as in Theorem \ref{main_result_localizing_map}. However, the constant $c$ differs significantly in the two results as when $f_X(x)$ is small, the constant in Theorem \ref{main_result_localizing_map} is of order \( ({\kappa f_X(x)} )^{-1/2} \), whereas in Theorem \ref{th:NN} it scales as \( (\kappa f_X(x))^{-1/d} \).

\subsection{CART-like regression tree}\label{s51}

We now consider general local regression trees for which each split is selected using a general cost function. In particular, the deviation inequality obtained below is valid for local regression maps that may depend on the whole dataset $(X_i,Y_i)_{i=1,\ldots,n }$ and not only on the covariates as in nearest neighbor algorithm. We call these trees "CART-like", since the CART algorithm is arguably the most important instance of such data-dependent regression trees, due to its wide fame and use in practice.

Let us introduce a general class of recursive data dependent trees.
For a given cell $V$, a split is characterized by two parameters $(p, u) \in S : = \{ 1,\ldots, d\}\times  (0,1)$. The resulting left and right child cells, $V(l)$  and $V(r)$, are such that for any $k \neq p$, $h_k(V(l))=h_k(V(r))  = h _k (V) $, and for $k=p$, $ h_k(V(l)) = h_k(V) u$ and $ h_k(V(r)) = h_k(V) (1-u) $. We also recall that $h_-(V) = \min_{k = 1,\ldots, d} h_k(V) $ and $h_+(V) = \max_{k = 1,\ldots, d} h_k(V) $.  With this notation, the split condition for $V$ to be $\beta$-shape regular can be expressed with the help of a restriction on the set of valid splits. Given $V$, let us define the set of $\beta$-shape regular splits as
$$ S_\beta(V) : = \{  (p,u)  \in S  \  : \  h_+( V(s)) \leq \beta h_-( V(s)),\,  \forall s  \in \left\{ l,r \right\}   \}. $$
We note that when $\beta \geq 2$, the  $S_\beta(V)$ cannot be empty. Splitting the largest side in the middle is always in $ S_\beta(V) $. Another restriction on the splits is needed to ensure a sufficient number of points. It is given by
$$ S_m(V) : = \{  (p,u) \in S  \  : \  n\mathbb P_n( V(s) ) \geq m , \, \forall s  \in \left\{ l,r \right\}   \} .$$
We do not need to fully specify the splitting criterion. When $S_m(V)\neq \emptyset$, the split in the cell $V$ is defined as a minimizer on $S_\beta(V)  \cap S_m(V)$, of a cost function $M_n$, given by
\begin{align*}
    M_n\, : \, S \times \mathcal R ([0,1]^d)  \to &\, \mathbb R \\
    ((p,u), V) \mapsto &\,  M_n ((p,u), V),
\end{align*}
where $ \mathcal R ([0,1]^d)$ is the set of hyper-rectangles in $S_X$. In case $S_m(V) = \emptyset$, no split is performed and the cell $V$ remains unchanged. The main strength of our analysis lies in the generality of the cost function, which can actually be any function that may depend or not on the sample. For instance, in CART-regression, the cost function depends on the sample and is defined as
$$ M_n ((p,u), V) = \frac{\sum_{i : X_i\in V(l) } (Y_i - \overline Y( V(l))) ^2 }{n\mathbb P_n (V(l) )  }  + \frac{\sum_{i : X_i\in V(r) } (Y_i - \overline Y(V(r))) ^2}{n\mathbb P_n (V(r) )   } $$
where $ \overline Y(V )= \sum_{i : X_i\in V } Y_i / (n\mathbb P_n (V ))$ for any cell $V$. 

\begin{algorithm}[h]
\begin{algorithmic}[h]
\Statex{\textbf{Input:} Sample  $(X_i, Y_i)_{i=1,\ldots, n} \subset [0,1]^d \times \mathbb R$, minimal number of points $m\in \{1,\ldots, n\} $, shape-regularity $\beta \geq 2$, cost function $M_n: S \times \mathcal R ([0,1]^d)  \to   \mathbb R$.
Let $V^{(0)} = \{[0,1]^d\} $ be the initial partition, made of one element (i.e. $| V^{(0)}|=1$).}
   \For{$j= 0,1,\ldots$}
\Statex\hspace{\algorithmicindent}{Let $V^{(j+1)} = \emptyset$} denote the partition at step $j+1$. The update is as follows: 
   \For{$k=  1,2,\ldots, | V^{(j)}|$} 
\Statex\hspace{\algorithmicindent}\hspace{\algorithmicindent}
(a) Whenever $S_m(V^{(j)}_k) \neq \emptyset $, define two children, $V(l)$ and $V(r)$, according to
$$\argmin_{(p,u) \in S_\beta(V^{(j)}_k)\cap S_m(V^{(j)}_k) } M_n( (p,u) , V^{(j)}_k) 
$$
\hspace{\algorithmicindent}\hspace{\algorithmicindent}{ If the above optimization problem has no solution, just pick $p$ as the largest side} 
\Statex\hspace{\algorithmicindent}\hspace{\algorithmicindent}{ and $u = 1/2$. Set $$V^{(j+1)} = \{ V^{(j+1)}, V (l), V (r)\}$$
}
\Statex \hspace{\algorithmicindent}\hspace{\algorithmicindent}{
(b) Whenever $S_m(V^{(j)}_k) = \emptyset $, child is same as parent. Set
  $$V^{(j+1)} = \{ V^{(j+1)}, V^{(j)}_k \}$$
}
\EndFor
\Statex\hspace{\algorithmicindent}{STOP if $ V^{(j+1)} = V^{(j)} $ (no valid split exists)}
   \EndFor
\Statex{Return the final partition elements $V^{(j+1)}$}
\end{algorithmic}
\caption{CART-like regression tree}
\label{alg:cart-like}
\end{algorithm}

By splitting on $S_\beta$ and $S_m$,  Algorithm \ref{alg:cart-like} ensures that two conditions are met when growing the tree. The first growing condition, that is the $\beta$-shape regularity of the cell, may not constitute a stopping criterion. Indeed, because $\beta\geq 2$, one can always split at the middle the largest side of the considered cell. The other growing condition on $m$ is easy to check in practice since it amounts to keep a cell as a leaf if and only if the number of data points belonging to that cell is greater than $m$ and strictly smaller than $2m$. As a consequence, one might modify classical algorithms, in the case precisely where the split proposed by the algorithm does not respect the $\beta-$shape-regularity condition for a prescribed value of $\beta$ or the other growing condition asking for sufficiently many points in the cells. The next theorem gives a deviation inequality for the associated regression map. 

\begin{theorem}\label{th_cart_like1}
Let $S_X=[0,1]^d$, $\delta \in (0,1/3) $, $n\geq 1$, $d\geq 1$, $\beta\geq 2$ and $m\in \{1,\ldots,  n\}$ such  that $m\geq  4\log(4(2n+1)^{2d}/\delta )$. Suppose that \ref{cond:density_X} and \ref{cond:epsilon} are fulfilled and that $g$ is $L$-Lipschitz. Let $\mathcal V$ be the local regression map obtained from a CART-like tree with input parameters $\beta $, $m$ and cost function $M_n$, then we have, with probability $1-3\delta$, for all $x \in S_X$,
    \begin{align*} 
            |  \hat g_\mathcal V (x) - g(x)  | \leq \sqrt { \frac{2 \sigma ^2 \log((n+1) ^{2d}/\delta) } {  m}} + L(\mathcal{V}(x)) \beta \sqrt{d} \left(\frac{5 m}{ n f_X(x) \kappa}\right)^{1/d}.
    \end{align*}
\end{theorem}

Note that the conditions on the value of $m$ are satisfied whenever $n$ is sufficiently large and $m\asymp n^a $, 
 for any $a\in (0,1)$. Notice that taking $m = n^{2/(d+2)}$ in the estimation bound of Theorem \ref{th_cart_like1} gives the optimal convergence rate $n^{-1/(d+2)}$, up a multiplicative logarithmic term. Moreover, such a value of $m$ allows the bound to be valid with a probability that grows to one polynomially in $n$, since the constraint $m \geq 4\log(4(2n+1)^{2d}/\delta )$ will be then satisfied.
In addition, such results and comments remain valid for the rate of convergence in sup-norm whenever the density $f$ is uniformly bounded from below by a positive constant independent of $n$. This is stated in the subsequent corollary.

\begin{corollary}\label{cor_cart_like1}
In Theorem \ref{th_cart_like1}, if the integer $m$ is chosen as $m \asymp n^{2/(d+2)} \log((n+1)^{2d}/\delta)^{d/(d+2)}$, then we have the following inequality for $n$ sufficiently large with probability at least $1-3\delta$, for all $x\in S_X$, 
   \begin{align*}
   | \hat g_{\mathcal V}(x) - g(x)|  \lesssim c \left( \dfrac{\log((n+1)^{2d}/\delta)}{n}\right)^{1/(d+2)},
   \end{align*}
    where $c =   \sqrt{2\sigma^2} + \beta L(\mathcal{V}(x)) \sqrt{d} \left[{5}/({\kappa f_X(x)}) \right]^{1/d}$.  In addition, whenever $S_X$ is bounded and $ f_X(x) \geq b >0 $ for all $x\in S_X$, we have for $n$ sufficiently large and with probability at least $1 - 3\delta $,
   \begin{align*}
      \sup_{x\in S_X} | \hat g_{\mathcal V}(x) - g(x)|  \lesssim c \left( \dfrac{\log((n+1)^{2d}/\delta)}{n}\right)^{1/(d+2)},
   \end{align*}
  where $c =   \sqrt{2\sigma^2} + \beta L \sqrt{d} \left[{5}/({\kappa b}) \right]^{1/d}$.
\end{corollary}

The previous result shows that CART-like regression trees are able to attain the optimal rate of convergence as soon as a simple constraint - restricting acceptable splits by a simple rule - is imposed during the tree construction. 

Interestingly, results presented in \cite{cattaneo2022pointwise} tend to indicate that such modifications are in general necessary for the classical CART algorithm to achieve a good pointwise - or uniform - behavior. More precisely, it is shown in \cite{cattaneo2022pointwise} that the use of CART is problematic for the estimation of a constant regression function,  measured with the sup-norm error. Indeed, its rate of convergence in dimension one is slower than any polynomial of the sample size $n$, with non-vanishing probability. In addition, the honest version of CART - i.e. when the prediction values among the cells use data that are independent of those used to construct the partition, see Definition 5.1 in \cite{cattaneo2022pointwise} -, is proved to be inconsistent with positive probability as soon as the tree depth is of order at least $\log(\log(n))$.  This is due to the fact that the splitting criterion produces leaves that are too small.  

Our results complete the picture drawn in \cite{cattaneo2022pointwise} by putting forward the fact that producing too small cells is \textit{the only problem} that can occur with the use of CART in dimension one. Indeed, any cell being $\beta$-shape-regular in dimension one, with $\beta=1$, Theorem \ref{th_cart_like1} shows that the only problem must come from the amont of data $m$ in the least populated cell. Indeed, if $m$ is of order $\log(n)$, then our deviation does not converge to zero when $\delta$ is fixed and the sample size goes to infinity. This is basically what happens in \cite{cattaneo2022pointwise}. In such case, we are indeed not able to prove the consistency of CART. On the contrary, when $m$ is of order $n^{2/(d+2)}$, Corollary \ref{cor_cart_like1} shows that our modified version of CART is rate optimal in sup-norm.





\section{Purely random trees}\label{sec_PRT}







We consider now purely random trees (PRT), that are built by successively refining a partition of the space, in a way that is independent from the initial sample $\mathcal D_n$. Before considering uniform, centered and Mondrian trees, we start by studying a key property of Lebesgue volume invariance which will be satisfied for the trees of interest. In this section we assume, for clarity, that $S_X= [0,1] ^d$ and we always take $x\in S_X$.


\subsection{Lebesgue volume invariance}\label{subsection_VIR}



To set up notations, let us describe a PRT locally around a point $x$ using the local maps framework introduced before. The tree is generated iteratively, and at each step $i$, for the cell $\mathcal V(x)$ containing $x$, a coordinate is selected according to a random variable $D_i\in \{ 1,\ldots,d\}$ and then the side of the cell in direction $D_i$, that we write $(a,b)$, $a<b$, is split into two intervals $(a,a+(b-a)S_i)$ and $(a+(b-a)S_i,b)$, thus defining two new cells $C_1$ and $C_2$. Consequently, each step $i$ consists in splitting a cell and depends on a pair of random variables $(D_i,S_i)$, that is independent from the dataset $\mathcal D_n$. After $N$ steps, we denote $\mathcal V(x)=\mathcal V (x, (D_i,S_i)_{i=1}^N)$.

In the following proposition, we state the remarkable fact that the Lebesgue volume of the cell can be expressed independently from the successive coordinate choices. We denote $\bar{S}_i$ the length reduction of the side $D_i$ of the considered cell at step $i$, that is either equal to $S_i$ or $1-S_i$ according to the fact that the coordinate $x_{D_i}$ is smaller or greater than $a+(b-a)S_i$ respectively.

\begin{proposition}[Lebesgue Volume Invariance] \label{prop_leb_vol_inv}
With the notations above, the following formula holds
\begin{equation*} \label{eq_vol}
    \lambda(\mathcal V(x,(D_i,S_i)_{i=1}^N))=\prod_{i=1}^N \bar{S}_i.
\end{equation*}
Assume in addition that for any $i$, the distribution of $S_i$ is symmetric around $1/2$, that is, $S_i \sim 1-S_i$. Then we get the following equality in distribution,
    \begin{equation*}\label{eq_dist_inv}
        \lambda (\mathcal V(x,(D_i,S_i)_{i=1}^N)) \sim  \prod_{i=1}^N S_i.
    \end{equation*}
\end{proposition}


Note that for centered or uniform random trees, the distributions of the $S_i$'s are indeed symmetric around $1/2$. Furthermore, for centered trees, $S_i=1/2$ almost surely, the Lebesgue volume of the cell containing $x$ after $N$ steps is equal to $1/2^N$. 

It is worth also noticing that actually, the formulas of Proposition \ref{prop_leb_vol_inv} are valid even if the random variables $D_i$ and $S_i$ depend on the dataset. Thus, the Lebesgue volume of the cell containing $x$ is independent from the direction choices as soon as the random vectors $(D_i)_{i=1}^N$ and $(S_i)_{i=1}^N$ are independent from each other, but not necessarily from the dataset. 

\subsection{Uniform random trees}\label{subsec_URT}

Let us first provide some deviation bounds for the diameter and volume of the localizing map built with uniform random trees.

\begin{proposition}\label{prop:diamURT2}
Consider  that $S_i=U_i$ are independent and uniformly distributed over $(0,1)$ and that $D_i$ are independent from each other and from the $U_i$'s and uniformly distributed over $\left\{ 1,\ldots,d\right\}$. Then, for $\mathcal V (x)= \mathcal V(x, (D_i,S_i)_{i=1}^N)$ and for any $\beta \geq 0$,
\begin{equation*}\label{eq_dev_diam_URT}
    \PP(\diam (\mathcal{V}(x))\geq \sqrt{d}e^{-N/d+N\beta}) \leq de^{-Nd\beta^2/4}.
\end{equation*} 
Moreover, for all $\beta \in (0,2/d)$ we have,
\begin{equation*}
    \PP(\diam (\mathcal{V}(x))\leq \sqrt{d}e^{-N/d - N\beta}) \leq de^{-Nd\beta^2/8}.
\end{equation*} 
\end{proposition}

\begin{proposition}\label{prop:diam_vol_URT2}
Consider  that $S_i=U_i$ are independent and uniformly distributed over $(0,1)$ and that $D_i$ are independent from each other and from the $U_i$'s and uniformly distributed over $\left\{ 1,\ldots,d\right\}$. Then, for $\mathcal V (x)= \mathcal V(x, (D_i,S_i)_{i=1}^N)$ and for any $\alpha >1$,
\begin{equation*}\label{eq_dev_vol_URT}
     \PP(\lambda(\mathcal{V}(x))\leq e^{-\alpha N}) \leq \left(\alpha e^{1-\alpha}\right)^N.
\end{equation*}
In addition, for any $\alpha \in (0,1)$,
\begin{equation*}
     \PP(\lambda(\mathcal{V}(x)) \geq e^{-\alpha N}) \leq \left(\alpha e^{1-\alpha}\right)^N.
\end{equation*}
\end{proposition}

\begin{corollary}\label{cor_diamvolURT}
When the number of splits goes to infinity, it holds that, almost surely, there exists $n_0\geq 1$ such that for all $N\geq n_0$,
$$ \sqrt d e^{-N/d-4\sqrt{N\log(N)/d}} \leq \diam(\mathcal{V}(x)) \leq \sqrt d e^{-N/d+2\sqrt{2N\log(N)/d}}$$
and
$$ e^{-N- 2\sqrt{N\log(N)}} \leq \lambda(\mathcal{V}(x))) \leq e^{-N + 2\sqrt{N\log(N)}}.$$
Moreover, if we denote the normalized diameter $\diam^{\#}(\mathcal{V}(x)) := \diam(\mathcal{V}(x)))/\sqrt{d}$, we obtain for $N$ large enough, 
\begin{equation*}
  e^{-2\sqrt{N\log(N)}(1+2\sqrt{d})}   \leq \dfrac{\diam^{\#}(\mathcal{V}(x))^d}{\lambda(\mathcal{V}(x)))} \leq e^{2\sqrt{N\log(N)}(1+\sqrt{2d})}. 
\end{equation*}
\end{corollary}
The previous results are valid in any dimension $d\geq 1$, but it is worth noting that in dimension one, the (normalized) diameter of any cell is always equal to its Lebesgue volume, $\diam(\mathcal{V}(x)))=\diam^{\#}(\mathcal{V}(x))=\lambda(\mathcal{V}(x)))$.

We deduce  the following high probability upper bound on the pointwise error of the resulting local map regression estimator. 

\begin{corollary}\label{cor_URT}
   Let $n\geq 1$, $d\geq 1$, $x\in S_X$ and $N=d\log(n)/(d+2)$. Suppose that $\mathcal V (x)= \mathcal V(x, (D_i,S_i)_{i=1}^N)$ is obtained from a uniform random tree as described in Proposition \ref{prop:diamURT2}. Under \ref{cond:epsilon} and \ref{cond:density_X}, suppose that $g$ is Lipschitz on $S_X$. Then there exists $C>0$, that only depends on the parameters of the problem but not on $n$, such that with probability $1$, there is $n_0$ such that for all $ n \geq n_0$,
\[
 | \hat g_{\mathcal V}(x) - g(x)| \leq
 C n^{-1/(d+2)}\sqrt{\log(n)} \, e^{2\sqrt{\log(n)\log\log (n)}}.
\]
\end{corollary}




From Corollary \ref{cor_diamvolURT},  we see that the local regression map estimator based on the uniform random partition achieves with probability tending to one an estimation error on a fixed point of the covariates space that is close to - but \textit{a priori} greater than - the optimal one, in the sense that for any $\varepsilon>0$, the estimation error is negligible compared to $n^{-1/(d+2)+\varepsilon}$ for $n$ sufficiently large.

The following negative result establishes that uniform trees are not shape-regular, thus indicating that the optimal rate of convergence may indeed not be achieved by the local estimator based on a uniform random tree.

\begin{proposition}\label{uniform tree not regular}
Uniform trees are not $\beta$-SR, i.e., for any \( N \geq d \) and any hyper-rectangle $\mathcal V (x)= \mathcal V(x, (D_i,S_i)_{i=1}^N)$ obtained from a uniform random tree as described in Proposition \ref{prop:diamURT2}, we have, with probability at least $1/11$,
$$\dfrac{h_+(\mathcal V (x) )}{h_-(\mathcal V (x) )} \geq e^{\sqrt{N/d}}.$$
\end{proposition}

While, in the above, the value $1/11$ can certainly be improved, we stress that our result implies that shape-regularity fails to happen on an event having positive probability (independent from $n$).

\subsection{Centered random trees}\label{subsec_centered}

In the case of centered random trees, the volume of the cell $\mathcal{V}(x)$ after $N$ steps is simply $  \lambda(\mathcal{V}(x)))= (1/2)^N$. The diameter of the localizing map behaves as follows.

\begin{proposition}\label{prop:diam_CRT}
Let $d\geq 2$ be an integer. Consider that $S_i=1/2 $ almost surely and that $D_i$ are independent from each other and uniformly distributed over $\left\{ 1,\ldots,d\right\}$. Then, for $\mathcal V (x)= \mathcal V(x, (D_i,S_i)_{i=1}^N)$ and for any $\alpha \in (0,1/d)$,
\begin{equation*}
    \PP(\diam(\mathcal{V}(x))\geq \sqrt{d}2^{-\alpha N}) \leq d\left(1-\frac{1-\beta}{d}\right)^N \beta^{-\alpha N},
\end{equation*}
where $\beta=(d-1)\alpha/(1-\alpha).$ Moreover, for any $\alpha \in (1/d,1)$, we have for the same $\beta$
\begin{equation*}
    \PP(\diam(\mathcal{V}(x))\leq \sqrt{d}2^{-\alpha N}) \leq d\left(1-\frac{1-\beta}{d}\right)^N \beta^{-\alpha N}.
\end{equation*}
\end{proposition}

We have the following corollary about the shape regularity of centered random trees.

\begin{corollary}\label{cor22}
When the number of splits goes to infinity, it holds that, almost surely, there exists $n_0\geq 1$ such that for all $N\geq n_0$,
$$ \sqrt{d} 2^{-N/d-2\sqrt{(d-1)N\log(N)/d^2}} \leq \diam(\mathcal{V}(x)) \leq \sqrt d 2^{-N/d+2\sqrt{(d-1)N\log(N)/d^2}}.$$
Moreover, if we denote the normalized diameter by $\diam^{\#}(\mathcal{V}(x)) := \diam(\mathcal{V}(x)))/\sqrt{d}$, we obtain for $N$ large enough,
\begin{equation*}
  2^{\, -2\sqrt{(d-1)N\log(N)}}   \leq \dfrac{\diam^{\#}(\mathcal{V}(x))^d}{\lambda(\mathcal{V}(x)))} \leq 2^{ \, 2\sqrt{(d-1)N\log(N)}}. 
\end{equation*}
\end{corollary}

In the same spirit as for uniform random tree, we give an upper bound on the pointwise error of the resulting local map regression estimator. 

\begin{corollary}\label{cor_UcR}
Let  $n\geq 1$, $d\geq 1$, $x\in S_X$ and $N=d\log(n)/\{(d+2)\log(2)\}$. Suppose that $\mathcal V (x)= \mathcal V(x, (D_i,S_i)_{i=1}^N)$ is obtained from a centered random tree as described in Proposition \ref{prop:diam_CRT}. Under \ref{cond:epsilon} and \ref{cond:density_X}, suppose that $g$ is Lipschitz on $S_X$, there exists $C>0$, that only depends on the parameters of the problem but not on $n$, such that with probability $1$, there is $N_0$ such that for all $ n \geq N_0$,
\[
 | \hat g_{\mathcal V}(x) - g(x)| \leq
 C n^{-1/(d+2)}e^{2\sqrt{\log(n)\log\log (n)}}.
\]
\end{corollary}

As for uniform random trees, the convergence rate is close to - but greater than - the optimal one, in the sense that for any $\varepsilon>0$, the estimation error is negligible compared to $n^{-1/(d+2)+\varepsilon}$ for $n$ sufficiently large.

We also include a negative result which establishes that centered trees are not shape-regular. 

\begin{proposition}\label{centered tree not regular}
Centered trees are not $\beta$-SR, i.e.,
for any \( N \geq d \) and any hyper-rectangle $\mathcal V (x)= \mathcal V(x, (D_i,S_i)_{i=1}^N)$ obtained from a centered random tree as described in Proposition \ref{prop:diam_CRT}, we have, with probability at least $1/14$,
$$\dfrac{h_+(\mathcal V (x))}{h_-(\mathcal V (x))} \geq 2^{\sqrt{N/d}}.$$
\end{proposition}
As for Proposition \ref{uniform tree not regular} above, the precise value $1/14$ does not play any crucial role here. The important fact is that Proposition \ref{centered tree not regular} shows that shape-regularity is violated on an event having positive probability (independent from $n$).

\subsection{Mondrian trees}\label{s43}

 A Mondrian process \textit{MP} is a process that generates infinite tree partitions of $S_X = [0,1]^d $. These partitions are built by iteratively splitting the different cells at random times, where both the timing and the position of the splits are determined randomly. Additionally, the probability that a cell is split depends on the length of its sides, and the probability of splitting a particular side is proportional to the length of that side. Once a side is selected, the exact position of the split is chosen uniformly along that side.
We can then define the pruned Mondrian process \textit{MP}($\lambda$). This version introduces a pruning mechanism that removes splits occurring after a specific time $\lambda > 0$, which is referred to as the lifetime.

Mondrian trees are studied in details in the paper \cite{minimaxmondrian}, along with various properties that help demonstrate that Mondrian trees are \(\beta\)-SR in probability.

\begin{proposition}\label{mondrian}
 Mondrian trees are \(\beta\)-SR in probability. More precisely, for any $x\in [0,1] ^d$, let  $\mathcal V (x)$ be the hyper-rectangle containing $x$ obtained from a \textit{MP}($\lambda $) tree. For $\delta \leq 1 - (1 - e^{-1})^d$, we then have, with probability at least \(1 - 2\delta\),
\[ \frac{h_+(\mathcal V (x))}{h_-(\mathcal V (x))} \leq \frac{5d \log(\delta/d)}{{\log(1-\delta)}}. \]
\end{proposition}

This implies that there is a constant $K_\delta >0 $ such that $h_+(\mathcal V (x)) \leq K_\delta \, h_-(\mathcal V (x))$ and thus ${h_+(\mathcal V (x))} /{h_-(\mathcal V (x))}~=~O_{\mathbb P}(1).$
As a consequence, we obtain optimal rates in probability  - that is, that hold with a fixed positive probability - for Mondrian trees. 


\section{Proofs} 


\subsection*{Proof of Theorem \ref{1}}

Let \( \mathbb P_{X_1^n} \) denote the conditional probability given \( X_1, \dots, X_n \). Let $\mathcal V = \{ \mathcal V (x) \, : \, x\in \mathbb R^d\}$ and define
$$ \mathcal G = \{ (\mathds 1 _{ A }(X_1), \dots,\mathds 1 _{ A}(X_n)) \, :\,  A \in \mathcal{A}  \}.$$
With this notation we have 
$$\sup_{x\in \mathbb R^d} \frac{\sum_{i=1}^n \varepsilon_i \mathds 1 _{ \mathcal V(x) }(X_i)}{\sqrt{\sum_{j=1}^n \mathds 1 _{ \mathcal V(x) }(X_j)}} \leq  \sup_{(g_1,\ldots, g_n) \in \mathcal G } \frac{\sum_{i=1}^n \varepsilon_i g_i }{\sqrt{ \sum_{j=1}^n g_j}} $$
Consequently, for all $t > 0$,
\begin{align*}
   \pr_{X_1^n}\left(\sup_{x\in \mathbb R^d} \dfrac{\sum_{i=1}^n \varepsilon_i \mathds 1 _{ \mathcal V(x) }(X_i)}{\sqrt{\sum_{j=1}^n \mathds 1 _{ \mathcal V(x) }(X_j)}} > t  \right) & \leq   \pr_{X_1^n}\left( \bigcup_{(g_1,\ldots, g_n) \in \mathcal G }  \left\{  \frac{\sum_{i=1}^n \varepsilon_i g_i }{\sqrt{ \sum_{j=1}^n g_j}} > t \right\}  \right) \\ &\leq \sum_{(g_1,\ldots, g_n) \in \mathcal G } \pr_{X_1^n}\left(\dfrac{\sum_{i=1}^n \varepsilon_i g_i}{\sqrt{ \sum_{j=1}^n g_j}} > t  \right).
\end{align*}
Moreover, since the conditional law of $\varepsilon_i$ with respect to $X_1,\ldots,X_n$ is sub-Gaussian with parameter $\sigma^2$, then  $\varepsilon_i g_i$ is sub-Gaussian under $\pr_{X_1^n}$, with parameter $\sigma^2 g_i^2$. Hence, ${\sum_{i=1}^n \varepsilon_i g_i}/{\sqrt{ \sum_{j=1}^n g_j}}$ is sub-Gaussian with parameter $\sigma^2 \sum_{i=1}^n g_i^2 / \sum_{j=1}^n g_j$ by independence. Moreover, $\sum_{i=1}^n g_i^2 = \sum_{i=1}^n g_i$ because $g_i\in \{ 0,1\} $. Hence, ${\sum_{i=1}^n \varepsilon_i g_i}/{\sqrt{ \sum_{j=1}^n g_j}}$ is sub-Gaussian with parameter $\sigma^2$ under $\pr_{X_1^n}$. Therefore, we obtain
\begin{align*}
    \pr_{X_1^n}\left(\sup_{x\in \mathbb R^d} \dfrac{\sum_{i=1}^n \varepsilon_i \mathds 1 _{ \mathcal V(x) }(X_i)}{\sqrt{\sum_{j=1}^n \mathds 1 _{ \mathcal V(x) }(X_j)}} > t  \right)  &\leq \sum_{(g_1,\ldots, g_n) \in \mathcal G } \exp\left( \dfrac{-t^2}{2\sigma^2}\right)  \leq  \mathbb S_\mathcal V(n) \exp\left( \dfrac{-t^2}{2\sigma^2}\right).
\end{align*}
   If we set $\delta = \mathbb S_\mathcal A(n) \exp\left( {-t^2} /{(2\sigma^2)}\right)$, we have $t = \sqrt{2\sigma^2 \log\left({\mathbb S_\mathcal A(n)}/{\delta}\right)}$. Finally, with probability $\pr_{X_1^n}$ at least equal to $1 - \delta$, we get
$$\sup_{x\in \mathbb R^d} \dfrac{\sum_{i=1}^n \varepsilon_i \mathds 1 _{ \mathcal V(x) }(X_i)}{\sqrt{\sum_{j=1}^n \mathds 1 _{ \mathcal V(x) }(X_j)}}  \leq \sqrt{2 \sigma^2 \log\left( \frac{\mathbb S_{\mathcal A}(n)}{\delta} \right)}.$$
   Since $\delta$ is independent of $(X_1,\ldots,X_n)$, we obtain the result by integrating with respect to $(X_1,\ldots,X_n)$.
\qed 

\subsection*{Proof of Theorem \ref{th:general}}

    Let \(x \in S_X\). We write the bias-variance decomposition 
$\hat g_{\mathcal V}(x) - g(x) = V + B$, where \begin{align*}
    V := \dfrac{\sum_{i=1}^n \varepsilon_i \mathds 1 _{ \mathcal V(x) }(X_i)}{\sum_{j=1}^n \mathds 1 _{ \mathcal V(x) }(X_j)} \qquad \text{and}\qquad
       B := \dfrac{\sum_{i=1}^n \left(g(X_i) - g(x)\right) \mathds 1 _{ \mathcal V(x) }(X_i)}{\sum_{j=1}^n \mathds 1 _{ \mathcal V(x) }(X_j)}.
\end{align*}
The inequality from Theorem \ref{1} gives, with probability at least $1 - 2\delta$, for all $x \in S_X $,
\begin{align*}
    |V| \leq \left({\sqrt{\sum_{j=1}^n \mathds 1 _{ \mathcal V(x) }(X_j)}} \right)^{-1} \sup_{x\in \mathbb R^d} \left| \dfrac{\sum_{i=1}^n \varepsilon_i \mathds 1 _{ \mathcal V(x) }(X_i)}{\sqrt{\sum_{j=1}^n \mathds 1 _{ \mathcal V(x) }(X_j)}} \right| \leq \dfrac{1}{\sqrt{n \PP_n(\mathcal{V}(x))}} \ \sqrt{2 \sigma^2 \log\left( \frac{\mathbb S_\mathcal A(n)}{\delta} \right)} .
\end{align*}
Using the inequality $\mathbb S_\mathcal A(n) \leq (n+1)^v$ we recover the first term of the stated bound. Furthermore, using the triangle inequality, we obtain that
\begin{eqnarray*}
    |B| &\leq& \dfrac{\sum_{i=1}^n \left|g(X_i) - g(x)\right| \mathds 1 _{ \mathcal V(x) }(X_i)}{\sum_{j=1}^n \mathds 1 _{ \mathcal V(x) }(X_j)} \\ &\leq& \dfrac{\sum_{i=1}^n \sup_{y \in \mathcal V(x)} |g(y) - g(x)| \mathds 1 _{ \mathcal V(x) }(X_i)}{\sum_{j=1}^n \mathds 1 _{ \mathcal V(x) }(X_j)} = \sup_{y \in \mathcal V(x)} |g(y) - g(x)|.
\end{eqnarray*}
Moreover, using the Lipschitz assumption, it follows that
$$|g(y) - g(x)| \leq  L(\mathcal V(x)) \|x-y\|_2  \leq L(\mathcal V(x)) \diam  (\mathcal V(x) ) , $$
which concludes the proof.
\qed

\subsection*{Proof of Theorem \ref{th2:general}}

Assume that the maximum of $\PP_n(\mathcal V(x))$ and $\PP(\mathcal V(x))$ is $\PP(\mathcal V(x)).$
We have, by assumption, for all $x\in S_X$,
$$ \frac{ n     \PP(\mathcal V(x)) }{2}  \geq 4 \log\left(\dfrac{4 (2n+1) ^v }{\delta} \right). $$
 Using that $1-1/\sqrt 2 \geq  2/3$, we deduce that
$$ \frac 2 3 \leq 1 -   \sqrt{ \frac{4 \log\left(\dfrac{4 (2n+1) ^v }{\delta} \right) }{n     \PP(\mathcal V(x))} }.$$
Hence, using Theorem \ref{th:vapnik_normalized}, we obtain that with probability $1-\delta$, for all $x\in S_X$,
$$ 
\PP_n(\mathcal V(x)) \geq \PP(\mathcal V(x)) \left(1-\sqrt { \frac{4 \log( 4 (2n+1) ^v / \delta) }{ n \PP (\mathcal V(x))} } \right)\geq \frac 2 3\PP (\mathcal V(x))  \geq \frac 2 3\kappa f_X(x) \lambda (\mathcal V(x) ).
$$
Now, if the maximum of $\PP_n(\mathcal V(x))$ and $\PP(\mathcal V(x))$ is $\PP_n(\mathcal V(x))$, then we have
$$ \PP_n(\mathcal V(x) ) \geq \PP(\mathcal V(x) )  \geq \kappa f_X(x) \lambda (\mathcal V(x) ).$$ 
Using Theorem \ref{th:general} and the previous inequality on $ \PP_n(\mathcal V(x) )$ yields the result.
\qed

\subsection*{Proof of Proposition \ref{link_beta_gamma}}

Let $A$ be a hyper-rectangle. We use the shortcut $h_- $ and $h_+$ for $h_- (A) $ and $ h_+(A)$, respectively. The first statement is a consequence of 
$ \diam (A) \leq \sqrt{d} h_+ $ 
and 
$ \lambda (A) \geq h_- ^d $, as using $\beta $-shape regularity, we obtain
$$ \diam (A)\leq \sqrt{d} \beta h_- \leq \sqrt{d} \beta \lambda (A) ^{1/d}  .$$
    The second statement can be obtained as follows. Since
$ \diam (A) \geq h_ + $
and  
$ \lambda (A)^{1/d} \leq h_+ ^{1 -1/d}  h_- ^{1/d}   $
we find
 \begin{align*}
      \gamma^{1/d} \geq \frac{ \diam (A)}{ \lambda (A)^{1/d} } \geq \frac{ h_ + }{  h_+ ^{1 -1/d}  h_- ^{1/d} } = \left(\frac{  h_ +}{   h_-  } \right)^{1/d}.
\end{align*}
\qed

\subsection*{Proof of Proposition \ref{contre_ex}}

Let $V_0 =  \mathcal V(0) $. Define
$$ W =  \frac{\sum_{i=1} ^n  (Y_i - g(X_i) ) \ind_{ V_0} (X_i)}{\sum_{i=1} ^n    \ind_{V_0 } (X_i)}  $$
and 
 $$B = \frac{\sum_{i=1} ^n  g(X_i)   \ind_{V_0 } (X_i)}{\sum_{i=1} ^n    \ind_{V_0} (X_i)}. $$
We have, since $g(0) = 0$, 
$$\hat g(0) - g(0) =  W + B , $$
and by conditional independence
$$ \EE [ (\hat g_{\mathcal V} (0 ) -  g(0 ))^2 ]  = \EE[W^2] + \EE [B^2]. $$
The lower bound for $W$ can be obtained relying on \ref{cond:epsilon}. We have
$$ \EE [W^2 |X_1,\ldots, X_n] = \frac{\sigma^2}{\sum_{i=1} ^ n \ind_{ V_0 } (X_i) } , $$
and then, taking the expectation and using Jensen's inequality, we get
$$  \EE [W^2 ] \geq \sigma^2  \EE \left[ \sum_{i=1} ^ n \ind_{V_0 } (X_i) \right]^{-1} = \sigma^2 ( n \lambda (V_0) ) ^{-1}.$$
To obtain the lower bound on $B$, we make use of \ref{cond:density_X}. Let  $V_1 = \prod_{k=1}^d [h_k/2 , h_k] \subset V_0$. We have, $a_0 = \sum_{i=1} ^n \ind_{V_0} (X_i)\geq \sum_{i=1} ^n \ind_{{ V_1}} (X_i) = a_1$. We are looking for a constant \( c > 0 \) such that \( a_1 \geq c a_0 \) in order to have
\begin{align*}
   B \geq \frac{\sum_{i=1} ^n g(X_i) \ind_{V_1} (X_i)}{\sum_{i=1} ^n \ind_{V_0} (X_i)}
    &\geq \frac{1}{2} ( h_1+ \dots + h_d ) \frac{a_1}{a_0} \\ &\geq \frac{c}{2}  \diam( V_0) .
\end{align*}
This implies that 
$$ \EE[B^2 ] \geq \EE [ \ind_{a_1\geq c a_0 } B^2]\geq \frac{c^2}{4}  \diam _1 ( V_0) ^2 \PP (  {a_1\geq ca_0 } ).$$
Let us look for such a constant \( c > 0\). From Theorem \ref{lemma=chernoff}, one has that with probability at least $1 - 2\delta = 1/2$,
$$ \frac{a_1}{a_0} \geq \frac{\PP({ V_1})}{\PP( V_0)} \frac{\left( 1 - \sqrt {2\log(4 ) / (n \PP( V_1)) }  \right )}{\left( 1 + \sqrt {3\log(4 ) / (n  \PP({  V_0 })})  \right )}.$$ Furthermore, note that $\lambda({ V_1}) = \prod_{k=1}^d h_k/2 = 2^{-d} \prod_{k=1}^d h_k = 2^{-d} \lambda( V_0)$ and $\PP({  V_k}) = \lambda({  V_k})$ for each  $k \in \{0,1\}$. Note also that we necessarily have $n \prod_{k=1}^d h_k \geq  2^{d+3} \log(4) \geq 3 \times 2^{d+1} \log(4) \geq 3 \times 4 \log(4).$ This ensures also that the numerator is positive. As a consequence, we find that, with probability at least $1/2$
\begin{align*}
    \frac{a_1}{a_0} &\geq 2^{-d} \, \dfrac{1 - \sqrt{\dfrac{2^{d+1} \log(4)}{n \prod_{k=1}^d h_k}}}{1+ \sqrt{\dfrac{3 \log(4)}{n \prod_{k=1}^d h_k}}} \geq 2^{-d} \dfrac{1 - 1/2}{1 + 1/2} = \dfrac{2^{-d}}{3} := c.
\end{align*}
Thus, we have obtained that
\begin{align*}
    \EE [ (\hat g_{\mathcal V} (0 ) -  g(0 ))^2 ]  &= \EE[W^2] + \EE [B^2] \\ 
    &\geq \dfrac{\sigma^2}{n \lambda (  V_0)} + \dfrac{ c ^2}{4} 
     \diam ( V_0)^2 \times \dfrac 1 2  \\
    &\geq \dfrac{\sigma^2}{n \lambda (  V_0)} + \dfrac{(c \gamma)^2}{8} \lambda ( V_ 0)^{2/d}. 
\end{align*}
where $\gamma = \overline \gamma ^{1/d} $. Let \( a_1 \) and \( a_2 \) be positive real numbers. By studying the function \( \psi : x \mapsto a_1x^{-d} + a_2 x^2 \) on \(\mathbb R_{>0}  \), we notice that \( \psi \) has global minimum achieved at \( x_m = (a_1d/(2a_2))^{1/(d+2)} \). This implies that \begin{align*}
\min_{x > 0} \psi (x) &\geq x_m^2 a_2 \left( \frac{ a_1 } {a_2x_m^{d +2 } }  + 1\right) 
\\ &= \left(\dfrac{a_1d}{2a_2}\right)^{2/(d+2)} a_2 \left( \dfrac{2}{d} +1  \right) 
\\ &= \left(\dfrac{a_2^{d/2}  a_1  d}{2}\right)^{2/(d+2)}  \left( \dfrac{2}{d} +1 \right)
\end{align*}
Now, setting \( a_1 = \sigma^2 n^{-1} \), \( a_2 = (c \gamma)^2/8 \), we find \begin{align*}
    \EE [ (\hat g_{\mathcal V} (0 ) -  g(0 ))^2 ] \geq \psi ( \lambda (\mathcal V)^{1/d}) \geq
    \left(\dfrac{\sigma^{2} d \, (c \gamma)^{d}}{2 (2\sqrt{2})^d \, n}\right)^{2/(d+2)} \left(1 + \dfrac{2}{d}\right)  = C_d ^2 \left(\dfrac{\bar\gamma \sigma^2}{n}\right)^{2/(d+2)}
\end{align*} where $$ C_d =  \sqrt{1 + \dfrac{2}{d}} \left(\dfrac{ d}{2}\right)^{1/(d+2)} \left( \frac 1 {2^d \sqrt{72}}\right)^{d/(d+2)}.$$
\qed

\subsection*{Proof of Theorem \ref{main_result_localizing_map}}

By assumption, there is $(a_-,a_+) $ such that $ 0< a _- \leq 1\leq   a _+ < +\infty $ and for all $x \in S_X$,
$$ \lambda(\mathcal V(x))  a_-  \leq \left( \frac{\log \left({(n+1)^v}/{\delta} \right)}{n} \right)^{d/(d+2)}\leq a_ + \lambda(\mathcal V(x)).$$
According to Theorem \ref{th2:general}, the $\gamma$-SR assumption, we obtain with probability at least $1 - 3\delta $, for all $x\in S_X$
\begin{align*}
        | \hat g_{\mathcal V}(x) - g(x)| &\leq  \sqrt{\frac{ 3 \sigma^2 \log\left( \frac{(n+1)^v}{\delta} \right) }{n  \kappa f_X(x) \lambda ( \mathcal V (x) )   }} + L(\mathcal V(x) ) \diam (\mathcal V(x) ) \\ 
        &\leq  \sqrt{\frac{ 3 \sigma^2 \log\left( \frac{(n+1)^v}{\delta} \right) }{n  \kappa f_X(x) \lambda ( \mathcal V (x) )   }} + L(\mathcal V(x) )\gamma^{1/d}  \lambda ( \mathcal V (x) )^{1/d} \\ &\leq  \sqrt{\frac{ 3 \sigma^2   \lambda ( \mathcal V (x) )  ^{(d+2)/d}a_+ ^{(d+2)/d}}{ \kappa f_X(x) \lambda ( \mathcal V (x) )   }} + L(\mathcal V(x) )\gamma^{1/d}  \lambda ( \mathcal V (x) )^{1/d} \\ &\leq  \left(  \sqrt{\frac{ 3 \sigma^2 a_+ ^{(d+2)/d} }{ \kappa f_X(x)}} + L(\mathcal V(x) )\gamma^{1/d} \right)  \lambda ( \mathcal V (x) )^{1/d} \\ &\leq \left( \sqrt{\frac{ 3 \sigma^2 a_+ ^{(d+2)/d}}{ \kappa f_X(x)}} + L(\mathcal V(x) )\gamma^{1/d} \right) 
 \left(\frac{\log\left( \frac{(n+1)^v}{ \delta} \right) }{ n} \right) ^{1/(d+2)} a_- ^{-1/d}.
\end{align*}
The result follows by taking care that $a_+ ^{(d+2)/d} \leq a_+^3 $  and $ a_- ^{-1/d}\leq a_-^{-1} $ which means that the universal constant in the upper bound can be taken as $ a_+^{3/2} / a_- $.






\qed

\subsection*{Proof of Theorem \ref{th:NN}}

        For any $x\in S_X$, define $ \tau(x)^d  = 2 k /  (n \kappa f_X(x) ) $  and check that  $\tau(x)^d \leq T_0^d$. Using \ref{cond:density_X} we obtain 
   \begin{align*}
    \forall x\in S_X,\qquad   n \PP(B (x, \tau(x)) ) \geq n \kappa f_X(x) \tau(x)^d   = 2  k .
   \end{align*}
   Next from Theorem \ref{th:vapnik_normalized}, and using that the set of all balls in $\mathbb R^d$, denoted by $\mathcal{A}$, has Vapnik dimension $d+1$ so that $\mathbb S_\mathcal A(2n)\leq (2n+1)^{(d+1)} $,  we deduce that with probability $1-\delta$,
   $$ \forall x\in S_X,\qquad  n \PP_n (B (x, \tau(x)) ) \geq n \PP(B (x, \tau(x)) ) - \sqrt { n \PP(B (x, \tau(x)) )  4\log( 4 (2n+1)^{(d+1)}  / \delta)   } . $$
   Note that $x\mapsto x - \sqrt{ x\ell} $ is increasing whenever $ x\geq \ell /4$. Since, by assumption on $k$, 
   $$\forall x\in S_X,\qquad  n \PP(B (x, \tau(x)) ) \geq  2k \geq 16\log( 4 (2n+1)^{(d+1)}  / \delta) \geq \log( 4 (2n+1)^{(d+1)}  / \delta).   $$
   We obtain that, with probability $1-\delta$,
   $$ \forall x\in S_X,\qquad    n \PP_n (B (x, \tau(x)) ) \geq 2k - \sqrt {8k  \log( 4  (2n+1)^{(d+1)}  / \delta)   }  .$$
   Now using again the assumption on $k$, $k \geq  {  8 \log( 4 (2n+1)^{(d+1)} / \delta)   },  $ which implies that with probability $1-\delta$
   $$ \forall x\in S_X,\qquad  n\PP_n(B (x, \tau(x)))\geq k .$$
   However, for each $x\in S_X$, $\hat \tau _{n,k}(x) $ is defined as the smallest such value of $\tau $. Therefore, we obtain that for all $x \in S_X$, $ \hat \tau_{n,k} (x)  \leq \tau (x)$ then, with probability $1-\delta$, 
   $$ \forall x\in S_X,\qquad  \hat \tau_{n,k}(x) ^d \leq \frac{ 2 k  }{ n \kappa f_X(x)  }.$$
The result then follows from applying Theorem \ref{th:general}. The variance term is obtained just noting that $n\mathbb P_n (\mathcal V (x) ) = k$ and $v = d+1 $ because the local map is valued in the collection of balls which VC dimension is given in \cite{wenocur1981some}. For the bias we use the Lipschitz condition and the inequality above since the \(\ell^2\)-diameter is twice the radius \(\hat \tau_{n,k}(x)\), which gives the upper bound with probability at least \(1 - 3\delta\).

\subsection*{Proof of Corollary \ref{corknn}}

By assumption, there is $(a_-,a_+) $ such that $ 0< a _- \leq 1\leq   a _+ < +\infty $ and $$  k \, a_-  \leq n^{2/(d+2)} \log((n+1)^{d+1}/\delta)^{d/(d+2)} \leq a_ + \, k.$$
According to Theorem \ref{th:NN}, we have the following inequalities with probability at least $1-3\delta$, for all $x\in S_X$, \begin{eqnarray*}
    && | \hat g_{\mathcal V}(x) - g(x)|  \\ &\leq& \sqrt{\frac{2\sigma^2\log ( {(n+1)^{d+1}/\delta })  a_+ }{ n^{2/(d+2)} \log((n+1)^{d+1}/\delta)^{d/(d+2)} }} +  2 L(\mathcal{V}(x)) \left( \frac{2 n^{2/(d+2)} \log((n+1)^{d+1}/\delta)^{d/(d+2)} }{n  \kappa f_X(x) a_- } \right)^{1/d} \\ &\leq& c \left( \dfrac{\log((n+1)^{d+1}/\delta)}{n}\right)^{1/(d+2)}  \dfrac{\sqrt{a_+}}{a_-}\end{eqnarray*}  where $c =    \sqrt{2\sigma^2} + 2 L(\mathcal{V}(x)) \left[2/({\kappa f_X(x)}) \right]^{1/d} .$ This choice of $k$ is valid for applying the theorem as long as $ 8 \log( 4 (2n+1)^{(d+1)} / \delta) \leq k = n^{2/(d+2)} \log((n+1)^{d+1}/\delta)^{d/(d+2)} \leq T_0^d n \kappa f_X(x) / 2,$ which is satisfied for $n$ sufficiently large. Moreover, if $f$ is bounded below uniformly on $S_X$ and $n$ is also sufficiently large, we can uniformly bound the upper bound $\sqrt{2\sigma^2} + 2 L(\mathcal{V}(x)) \left[{2}/({\kappa f_X(x)}) \right]^{1/d}$ by $\sqrt{2\sigma^2} + 2 L \left[{2}/({\kappa b}) \right]^{1/d}.$ \qed

\subsection*{Proof of Theorem \ref{th_cart_like1}}

The proof follows from an easy application of the next result which is stated for general local regression maps. 

\begin{theorem}\label{th_cart_like}
Let $S_X=[0,1]^d$, $\delta \in (0,1/3) $, $n\geq 1$, $d\geq 1$, and $m\geq 4\log(4(2n+1)^{2d}/\delta )$. Suppose that \ref{cond:density_X} and \ref{cond:epsilon} are fulfilled and that $g$ is $L$-Lipschitz. Let $\beta \geq 2$ and suppose that $\mathcal V$ is a local regression map valued in the set of hyper-rectangles contained in $S_X$, for all $V\in \left\{ \mathcal V(x) \, :\, x\in \mathbb R^d \right\}$,
    \begin{align*}
        h_+(V) \leq \beta h_- (V)
\quad \text{and}\quad 
       n \mathbb P_n( V) \geq  m,
    \end{align*}
    then we have, with probability $1-3\delta$, for all $x \in S_X$,
    \begin{align*} 
            |  \hat g_\mathcal V (x) - g(x)  | \leq \sqrt { \frac{2 \sigma ^2 \log((n+1) ^{2d}/\delta) } {  m}} + L(\mathcal{V}(x)) \beta \sqrt{d} \left(\frac{5 m}{ n f_X(x) \kappa}\right)^{1/d}.
    \end{align*}
\end{theorem}

    Note that,  when growing the tree, the constraint $h_+(V) \leq \beta h_- (V) $ can never be a stopping criterion because one can always select the largest side and split it in the middle.
    When the tree is fully grown according to the prescribed rules,  acceptable splits are no longer possible. Therefore any $V $ satisfies 
    $$ 2 m \geq   n\mathbb P_n(V) \geq  m.$$
Since the Vapnik dimension of hyper-rectangles is $ v = 2d$, using Assumption \ref{cond:density_X} and Theorem \ref{th:vapnik_normalized}, we obtain with probability at least $1 - \delta$,
$$f_X(x) \kappa h_-^d \leq \PP(V) \leq \dfrac{4}{n} \log\left(\dfrac{4(2n+1)^{2d}}{\delta} \right) + 2 \PP_n(V) \leq \frac{m}{n} + \frac{4m}{n} = \frac{5m}{n}.$$
In addition,
$$\diam(V) \leq \sqrt{d} h_+ \leq \sqrt{d} \beta h_- \leq  \sqrt{d} \beta \left(\dfrac{5 m}{n f_X(x) \kappa}\right)^{1/d}.$$
It remains to apply Theorem \ref{th:general} and to use that $n \PP_n(V) \geq m$ for the variance term to get the stated result.
\qed

\subsection*{Proof of Corollary \ref{cor_cart_like1}}

It is sufficient to reason as in Corollary \ref{corknn} by applying Theorem \ref{th_cart_like1} and by bounding \(f_X\) from below by \(b > 0\).
\qed

\subsection*{Proof of Proposition \ref{prop_leb_vol_inv}}

Recall that $S_X= [0,1] ^d$, so each side length of the initial cell is equal to one. For any $k\in \{1,\ldots,d\}$, we have the following formula for the length $h_k$ of the side $k$ of the cell $\mathcal V(x,(D_i,S_i)_{i=1}^N)$,
    \[
    h_k= \prod_{i=1}^n  \tilde S_i^{B_i^{(k)}},
    \]
    where $B_i^{(k)}=\mathbb{I}_{D_i=k}$. Taking the product over $k$ gives
    \begin{align*}
         \prod_{k=1}^d h_k 
         &= \prod_{k=1}^d \prod_{i=1}^n \tilde S_i^{B_i^{(k)}} 
         =  \prod_{i=1}^n \tilde S_i^{\sum_{k=1}^d B_i^{(k)}}.
    \end{align*}
    The result follows by noting that $\lambda(\mathcal V(x,(D_i,S_i)_{i=1}^N))=\prod_{k=1}^d h_k$ and that for any $i$, $\sum_{k=1}^d B_i^{(k)} = 1$, the latter identity simply corresponding to the fact that exactly one side of the cell is split at each step.
\qed

\subsection*{Proof of Proposition \ref{prop:diamURT2}}

        First notice that, by a union bound and symmetry in the directions, we have
\[
 \pr(\diam (\mathcal{V}(x))\geq t) \leq d\pr \left(h_1 \geq \frac{t}{\sqrt{d}}\right).
\]
Furthermore, by denoting $B_i^{(1)}=\ind_{D_i=1}$ as in the proof of Proposition \ref{prop_leb_vol_inv}, we get for any $r\in (0,1)$ and $\lambda>0$,
\begin{align*}
    \pr(h_1 \geq r^N) & = \pr\left(\prod_{i=1}^N U_i^{B_i^{(1)}} \geq r^N\right)\\
    & \leq \mathbb E\left[\left(\frac{\prod_{i=1}^N U_i^{B_i^{(1)}}}{r^N}\right)^\lambda \right]=\left(\frac{\mathbb E\left[{U_1^{\lambda B_1^{(1)}}}\right]}{r^{\lambda}}\right)^N. 
\end{align*}
It holds
\[
 \mathbb E\left[{U_1^{\lambda B_1^{(1)}}}\right]= \frac{1}{d(1+\lambda)}+1-\frac{1}{d}.
\]
Hence,
\[
 \pr(h_1 \geq r^N) \leq \left(\frac{1}{d(1+\lambda)}+1-\frac{1}{d}\right)^N r^{-\lambda N}.
\]
First note that it suffices to optimize the bound for $N=1$. Let us denote 
\[
Q(\lambda)= \frac{1}{d(1+\lambda)}+1-\frac{1}{d}
\]
and 
\[
h(\lambda)= Q(\lambda) r^{-\lambda}.
\]
Denote $r=(1/e)^{1/d-\beta}$ for $\beta>0$, then
\begin{align*}
    h(\lambda)
    & =   \exp\left(\lambda\left(\frac{1}{d}-\beta\right)+\log\left(1-\frac{\lambda}{d(1+\lambda)}\right)\right)\\
    & \leq  \exp\left(\lambda\left(\frac{1}{d}-\beta\right)-\frac{\lambda}{d(1+\lambda)}\right) \\
    & \leq  \exp\left(\lambda\left(\frac{1}{d}-\beta\right)-\frac{\lambda(1-\lambda)}{d}\right) \\
    & =   \exp\left(-\lambda\left(\beta-\frac{\lambda}{d}\right)\right).
\end{align*}
By taking $\lambda= d\beta/2$, we get
\[
 \pr(h_1 \geq e^{N(\beta-1/d)}) \leq e^{-d\beta^{2}N/4}.
\]
Then for all $\beta > 0$, we obtain
\[
 \pr(\diam (\mathcal{V}(x))\geq \sqrt{d}e^{N(\beta-1/d)}) \leq d\pr \left(h_1 \geq e^{N(\beta-1/d)}\right) \leq d e^{-d\beta^{2}N/4}.
\]
We now consider the lower bound. We proceed in the same way as before. By a union bound and symmetry in the directions, we have
\[
 \pr(\diam (\mathcal{V}(x)))\leq t) \leq d\pr \left(h_1 \leq \frac{t}{\sqrt{d}}\right).
\]
Furthermore, for any $(r, \lambda)\in (0,1)^2$,
\begin{align*}
    \pr(h_1 \leq r^N) & = \pr\left(\prod_{i=1}^N U_i^{B_i^{(1)}} \leq r^N\right) = \pr\left(\prod_{i=1}^N U_i^{-\lambda B_i^{(1)}} \geq r^{-\lambda N} \right) \\
    & \leq \mathbb E\left[\left(\frac{\prod_{i=1}^N U_i^{B_i^{(1)}}}{r^N}\right)^{-\lambda} \right]=\left(\frac{\mathbb E\left[{U_1^{-\lambda B_1^{(1)}}}\right]}{r^{-\lambda}}\right)^N. 
\end{align*}
It holds
\[
 \mathbb E\left[{U_1^{-\lambda B_1^{(1)}}}\right]= \frac{1}{d(1-\lambda)}+1-\frac{1}{d}.
\]
Hence,
\[
 \pr(h_1 \leq r^N) \leq \left(\frac{1}{d(1-\lambda)}+1-\frac{1}{d}\right)^N r^{\lambda N}.
\]
Without loss of generality, we can optimize the bound for $N=1$. Define 
$$Q(\lambda)= \frac{1}{d(1-\lambda)}+1-\frac{1}{d} = 1 + \dfrac{\lambda}{d(1-\lambda)},$$
and
$$ h(\lambda)= Q(\lambda) r^{\lambda}.$$
Set $r=(1/e)^{1/d+\beta}$ for $\beta>0$, then for all $\lambda \in (0, 1/2)$,
\begin{align*}
    h(\lambda)
    & =   \exp\left(-\lambda\left(\frac{1}{d}+\beta\right)+\log\left(1 + \dfrac{\lambda}{d(1-\lambda)}\right)\right)\\
    & \leq  \exp\left(-\lambda\left(\frac{1}{d}+\beta\right)+ \dfrac{\lambda}{d(1-\lambda)}\right) \\
    & \leq  \exp\left(-\lambda\left(\frac{1}{d}+\beta\right)+ \frac{\lambda(1+2\lambda)}{d}\right) \\
    & =   \exp\left(-\lambda\left(\beta-\frac{2\lambda}{d}\right)\right),
\end{align*}
where in the second inequality we used the fact that $(1-\lambda)^{-1} \leq 1+2\lambda$ for $\lambda \in (0,1/2)$. By taking $\lambda= d\beta/4 \in (0, 1/2)$, we get
\[
 \pr(h_1 \leq e^{-N(\beta+1/d)}) \leq e^{-d\beta^{2}N/8}.
\]
Then for all $\beta \in (0, 2/d)$,
\[
 \pr(\diam (\mathcal{V}(x)))\leq \sqrt{d}e^{-N(\beta+1/d)}) \leq d\pr \left(h_1 \leq e^{-N(\beta+1/d)}\right) \leq d e^{-d\beta^{2}N/8}.
\]

\qed

\subsection*{Proof of Proposition \ref{prop:diam_vol_URT2}}

As in the proof of Proposition \ref{prop:diamURT2}, we optimize along some polynomial moments controlling the deviation probability of interest. We have $\lambda(\mathcal V(x))=\prod_{i=1}^N U_i$, which gives, for any $\alpha>1$, $\lambda \in (0,1)$,
\[
\PP((\lambda(\mathcal V(x)))^{-1}\geq e^{N\alpha})\leq \EE\left[\prod_{i=1}^N U_i^{-\lambda}\right]e^{-N\alpha \lambda} = \left(\frac{e^{-\alpha \lambda}}{1-\lambda}\right)^N.
\]
By taking $\lambda =1 -1/\alpha$, we get
\[
\PP((\lambda(\mathcal V(x)))^{-1}\geq e^{N\alpha})\leq (\alpha e^{1-\alpha})^N.
\]
Moreover for any $\lambda>0$, $\alpha \in (0,1)$,
\[
\PP(\lambda(\mathcal V(x)) \geq e^{-N\alpha})\leq \EE\left[\prod_{i=1}^N U_i^{\lambda}\right]e^{N\alpha \lambda} = \left(\frac{e^{\alpha \lambda}}{1+\lambda}\right)^N.
\]
By taking $\lambda =1/\alpha - 1$, this gives
\[
\PP(\lambda(\mathcal V(x))\geq e^{-N\alpha})\leq (\alpha e^{1-\alpha})^N.
\]
\qed

\subsection*{Proof of Corollary \ref{cor_diamvolURT}}

We will use the Borel Cantelli lemma together with the inequalities obtained in theorems \ref{prop:diamURT2} and \ref{prop:diam_vol_URT2}. To prove the upper bound on the diameter, we provide values $\beta_N$ leading to small enough probabilities. More precisely, by taking $\beta_N=2\sqrt{2\log(N)/(dN)}$, we get $e^{-Nd\beta_N^2/4}=N^{-2}$.   Then  
    \begin{equation*}
    \PP\left(\diam (\mathcal{V}(x))\geq \sqrt{d}e^{N(-1/d+\beta_N)}\right) \leq \dfrac{d}{N^2}.
\end{equation*} 
The Borel-Cantelli lemma then gives $$\PP\left(\liminf_{N \to + \infty} \left\{ \diam (\mathcal{V}(x)) \leq \sqrt{d} e^{N(-1/d + \beta_N)} \right\}\right) = 1.$$
This means that almost surely, beyond a certain rank, we have $$\diam (\mathcal{V}(x)) \leq \sqrt{d} e^{N(-1/d + \beta_N)}.$$
For the lower bound on the diameter, we proceed in the same way with the choice $\tilde \beta_N=4\sqrt{\log(N)/(dN)}$, or equivalently $e^{-Nd \tilde \beta_N^2/8}=N^{-2}$. We deduce that, almost surely, beyond a certain rank, $$\diam (\mathcal{V}(x)) \geq \sqrt{d} e^{-N(1/d + \tilde \beta_N)}.$$
Now regarding the volume, we set $\alpha_N=1+2\sqrt{\log(N)/N}$ and we obtain \[
    \left(\alpha_N e^{1-\alpha_N}\right)^N = \exp\left( {-2\sqrt{N\log(N)}+N\log\left(1+2\sqrt{\log(N)/N}\right)}\right).\]
As $\log\left(1+2\sqrt{\log(N)/N}\right)=2\sqrt{\log(N)/N} - 2\log(N)/N+ O\left((\log(N)/N)^{3/2}\right)$, we get $$\left(\alpha_N e^{1-\alpha_N}\right)^N = \exp \left( -2\log(N) + O\left(\log(N)^{3/2} / \sqrt{N} \right) \right) \sim 1/N^2.$$
The Borel-Cantelli lemma gives us that almost surely, beyond a certain rank $n_0$, we have
$$\forall N \geq n_0, \qquad \frac{ \lambda(\mathcal{V}(x))) }{ e^{-N- 2\sqrt{N\log(N)}}} \geq 1.$$
For the upper bound on the volume, we set for $N \geq 9$, $\tilde \alpha_N=1 - 2\sqrt{\log(N)/N} \in (0,1)$ and we obtain \[
    \left(\tilde \alpha_N e^{1-\tilde \alpha_N}\right)^N = \exp\left( {2\sqrt{N\log(N)}+N\log\left(1-2\sqrt{\log(N)/N}\right)}\right).\]
As $\log\left(1-2\sqrt{\log(N)/N}\right)=-2\sqrt{\log(N)/N} - 2\log(N)/N+ O\left((\log(N)/N)^{3/2}\right)$, we get $$\left(\tilde \alpha_N e^{1-\tilde \alpha_N}\right)^N = \exp \left( -2\log(N) + O\left(\log(N)^{3/2} / \sqrt{N} \right) \right) \sim 1/N^2.$$
The Borel-Cantelli lemma gives us that almost surely, beyond a certain rank $n_0$, we have
$$\forall N \geq n_0, \qquad \frac{ \lambda(\mathcal{V}(x))) }{ e^{-N + 2\sqrt{N\log(N)}}} \leq 1.$$
Finally, the last inequality stated in Corollary \ref{cor_diamvolURT} comes readily by using the two previous inequalities on the diameter and the volume.
\qed

\subsection*{Proof of Corollary \ref{cor_URT}}

Firstly, since the local regression map is obtained from a tree construction, each element in $\mathcal V ( x) : = \mathcal V ( x, (D_i,S_i)_{i=1}^N )$ is a rectangle. Hence, in light of \cite{wenocur1981some}, it holds that the image of the resulting local map is included in the set of rectangles, that has VC dimension $v = 2d$. Hence, the local map is indeed VC. 

Secondly, according to Theorem \ref{th2:general}, applied pointwise for $x\in S_X$ and with $\delta = (n+1)^{-2}$, we find that whenever $n \mathbb P (\mathcal{V}(x) ) / \log(n) \to \infty$, it holds that
$$ \sum_{n\geq 1} \mathbb P ( |\hat g_{\mathcal V}(x) - g(x)| > v_n ) < \infty,
$$
where 
$$v_n = \sqrt{ 3 \sigma^2 \log( {(n+1)^{v+2}}  ) / (n  \kappa f_X(x) \lambda ( \mathcal V(x)) )}  + L(\mathcal V(x)) \diam (\mathcal V(x)).$$
Applying the Borel Cantelli Lemma, we get that with probability $1$, for $n$ large enough,
$$|\hat g_{\mathcal V}(x) - g(x)| \leq  \sqrt{\frac{ 3 \sigma^2 \log\left( {(n+1)^{v+2}} \right)}{n  \kappa f_X(x) \lambda ( \mathcal V(x)) }} + L(\mathcal V(x)) \diam (\mathcal V(x)) .$$
Thirdly, from Corollary \ref{cor_diamvolURT} and using that $ N = d\log(n)/(d+2)$, with probability $1$, for a sufficiently large \(n\), we have
\begin{align*}
& \lambda(\mathcal{V}(x)) \geq e^{-N - 2 \sqrt{N \log(N)}} \geq n^{-d/(d+2)} e^{- 2 \sqrt{\log(n) \log(\log(n)) }  },
\end{align*}
where we have used that \(N \leq \log(n)\). Using \ref{cond:density_X}, it follows that
$$ \mathbb P (\mathcal{V}(x) ) \geq \kappa f_X(x) \lambda(\mathcal{V}(x))  \geq \kappa f_X(x) n^{-d/(d+2)} e^{- 2 \sqrt{\log(n) \log(\log(n)) }  } .$$
As a consequence,
$$ n \mathbb P (\mathcal{V}(x) ) \geq \kappa f_X( x) n ^{ 2 /(d+2) }       e^{- 2 \sqrt{\log(n) \log(\log(n)) }  }  .$$
Hence, we get that with probability $1$, $ n \mathbb P (\mathcal{V}(x) ) / \log(n) \to \infty$. This ensures the $(\delta, n)$-large hypothesis, in order to apply Theorem \ref{th2:general}.

Fourthly, by putting together the second and third point from above, we have the following inequality, with probability $1$, for $n $ large enough and $\delta = (n+1)^{-2}$,
$$|\hat g_{\mathcal V}(x) - g(x)| \leq \sqrt{\frac{ 3 \sigma^2 \log\left( (n+1)^{v+2} \right)}{n  \kappa f_X(x) \lambda ( \mathcal V(x)) }} + L(\mathcal V(x)) \diam (\mathcal V(x)).$$
This gives in virtue of Corollary \ref{cor_diamvolURT}
$$| \hat g_{\mathcal V}(x) - g(x)| \leq \sqrt{\frac{ 3 \sigma^2 \log\left( (n+1)^{v+2}\right)}{n \kappa f_X(x) e^{-N - 2\sqrt{N \log(N)}}}} + L(\mathcal V(x)) \sqrt d e^{-N/d + 2\sqrt{2N \log(N)/d}}.$$
Recalling that \( N = d\log(n)/(d+2) \), we obtain
\begin{align*}
    &| \hat g_{\mathcal V}(x) - g(x)| \\ & \leq n^{-1/(d+2)} e^{\sqrt{N \log(N)}} \sqrt{\frac{ 3 \sigma^2 \log\left( (n+1)^{v+2} \right)}{\kappa f_X(x)}} + n^{-1/(d+2)} L(\mathcal V(x)) \sqrt d e^{2\sqrt{2N \log(N)/d}}.
\end{align*}
Then
\[
| \hat g_{\mathcal V}(x) - g(x)| \leq n^{-1/(d+2)} e^{C_d \sqrt{N \log(N)}} \left( \sqrt{\frac{ 3 \sigma^2 \log\left( (n+1)^{v+2} \right)}{\kappa f_X(x)}} + L(\mathcal V(x)) \sqrt d \right),
\]
where \(C_d = \max(1, \sqrt{8/d})\). But since $\log(n+1) \leq 2 \log(n)$ for $n\geq2$, we have

\[
\sqrt{\frac{ 3 \sigma^2 \log\left( (n+1)^{v+2}\right)}{\kappa f_X(x)}} = \sqrt{\frac{ 3 \sigma^2 (v+2) \log(n+1)}{\kappa f_X(x)}} \leq \sqrt{\frac{ 6 \sigma^2 (v+2) \log(n)}{\kappa f_X(x)}}.
\]
Then we set $$C = \sqrt{\frac{ 6 \sigma^2 (v+2)}{\kappa f_X(x)}} + L(\mathcal V(x)) \sqrt d = \sqrt{\frac{ 6 \sigma^2 (d+1)}{\kappa f_X(x)}}+ L(\mathcal V(x)) \sqrt d.$$
Additionally, since \(N \leq \log(n)\), we have 
\[
N \log(N) \leq \frac{d}{d+2} \log(n) \log(\log(n)),
\]
so we set \(c_d = \sqrt{\frac{d}{d+2}} C_d = \max\left(\sqrt{\frac{d}{d+2}}, \sqrt{\frac{8}{d+2}}\right) \leq 2 \) to obtain the desired inequality.
\qed

\subsection*{Proof of Proposition \ref{uniform tree not regular}}

At each stage, for each terminal leaf, draw uniformly $D_i $ in $\{1,\ldots, d\}$ as well as a uniform random variable $U_i$. Then we divide the cell according to coordinate $k = D_i$. The corresponding length $h_k(\mathcal V (x) )$ is then updated into $h_k(\mathcal V (x) ) U_i$ and $h_k(\mathcal V (x) ) (1-U_i)$. Note that $1-U_i$ is still uniformly distributed. As a consequence, for a given leaf, after $N$ stages,  the $k$-th length has the following representation 
$$ h _ k(\mathcal V (x) ) = U_1^{B_1^{(k)}}\times \ldots \times U_N^{B_N^{(k)}} = \exp\left( \sum_{i=1} ^ N B_i^{(k)} \log(U_i) \right)  $$
where $B_i^{(k)}  = \ind_{D_i = k }$. It follows that
\begin{align*}
 &h_+(\mathcal V (x) ) = \exp\left( \max_{k=1,\ldots, d}  \sum_{i=1} ^ N B_i^{(k)} \log(U_i) \right), \\
 &h_-(\mathcal V (x) ) =  \exp\left(\min_{k=1,\ldots, d}  \sum_{i=1} ^ N B_i^{(k)} \log(U_i) \right), 
\end{align*}
and the expression of the ratio is
\begin{align*}
 h_+(\mathcal V (x) )/h_-(\mathcal V (x) )  &= \exp \left(   \max_{1 \leq k,j \leq d} \sum_{i=1} ^N (B_i^{(k)} - B_i^{(j)} ) E_i \right) \end{align*}
 where $E_i = - \log(U_i)$ follows an exponential distribution with parameter 1.

By denoting $V_i^{k,j} = B_i^{(k)} - B_i^{(j)}$, we get
\[V_i^{k,j} = \begin{cases} 
1 & \text{with probability } 1/d \\
0 & \text{with probability } 1 - 2/d \\
-1 & \text{with probability } 1/d 
\end{cases}.\]
Note that the variables \((V_i^{k,j})_{i=1}^N\) are mutually independent because the \((D_i)_{i=1}^N\) are independent. Furthermore, since the \(U_i\)'s are independent of the \(V_i\)'s, the \(V_i^{k,j}\)'s are independent of the \(E_i\)'s.
Let \(Z_{k,j} = \sum_{i=1}^{N} V_i^{k,j} E_i\) such that
\[
\frac{h_+(\mathcal V (x) )}{h_-(\mathcal V (x) )} = \exp \left( \max_{1 \leq k,j \leq d} Z_{k,j} \right).
\]
Note that \(Z_{k,j} = - Z_{j,k}\) and \(Z_{k,k} = 0\), which gives
\[
\max_{1 \leq k,j \leq d} Z_{k,j} = \max_{1 \leq k < j \leq d} | Z_{k,j} |
\]
and thus the formula 
\[
\frac{h_+(\mathcal V (x) )}{h_-(\mathcal V (x) )} = \exp \left( \max_{1 \leq k < j \leq d} | Z_{k,j} | \right).
\]
By using the Paley-Zygmund inequality to $Z_{k,j}^2$, we get for all $\theta \in (0,1)$,
\[
\PP\left(|Z_{k,j}| \geq \sqrt{\theta} \sqrt{\mathbb{E}(Z_{k,j}^2)} \right) \geq (1-\theta)^2 \frac{\mathbb{E}(Z_{k,j}^2)^2}{\mathbb{E}(Z_{k,j}^4)}.
\]
We therefore seek to calculate the 2nd and 4th moments of \(Z_{k,j}\). Since \(Z_{k,j}^2 = \sum_{i \neq \ell} V_i^{k,j} V_\ell^{k,j} E_i E_\ell + \sum_{i=1}^N V_i^{k,j \, 2} E_i^2 \), \(\mathbb{E}(V_i^{k,j}) = 0\) and by independence along the subscripts, we obtain 
$$\mathbb E(Z_{k,j}^2) = \sum_{i=1}^N \mathbb{E}((V_i^{k,j})^2) \mathbb{E}(E_i^2) = N \times \frac{2}{d} \times 2 = \frac{4N}{d}.$$
Moreover, according to Lemma \ref{lemme moment} applied to \( M_i := V_i^{k,j} E_i \), we obtain 
\begin{eqnarray*}
    \mathbb{E}(Z_{k,j}^4) &=& N \mathbb{E}(M^4) + 3N (N-1) \mathbb{E}(M^2)^2 \\ &=& N \mathbb{E}((V_i^{k,j})^4) \mathbb{E}(E_i^4) + 3N (N-1) \mathbb{E}((V_i^{k,j})^2)^2 \mathbb{E}(E_i^2)^2.
\end{eqnarray*}
Indeed, it is easily checked that the variables \( (M_i)_{i=1}^N \) are centered and independent, due to the independence between the elements of the collections \( (V_i^{k,j})_{i=1}^N \) and \( (E_i)_{i=1}^N \) and the fact that the $V_i^{k,j}$ are centered. Basic calculations then give
\begin{eqnarray*}
    \mathbb{E}(Z_{k,j}^4) &=& N \mathbb{E}({V_i^{k,j \, 4}}) \mathbb{E}(E_i^4) + 3N (N-1) \mathbb{E}({V_i^{k,j \, 2}})^2 \mathbb{E}(E_i^2)^2 \\ &=& N \times \dfrac{2}{d} \times  4! + 3N(N-1) \left( \dfrac{2}{d}\right)^2 \times 2^2 \\ &=& \dfrac{48N}{d^2} (d+N-1).
\end{eqnarray*}
Consequently, we get
\[
\frac{\mathbb{E}(Z_{k,j}^2)^2}{\mathbb{E}(Z_{k,j}^4)} = \frac{16N^2}{d^2} \times \frac{d^2}{48N (d + N -1)} = \frac{N}{3(d + N -1)}
\]
and thus, for all \(\theta \in (0,1)\),
\[
\PP\left(|Z_{k,j}| \geq \sqrt{\theta} \sqrt{\frac{4N}{d}}\right) \geq (1-\theta)^2  \frac{N}{3(d + N -1)}.
\]
In particular, for \(N \geq d\), we have $3(d + N - 1) \leq 6N$, which gives
\[
\PP\left(|Z_{k,j}| \geq \sqrt{\theta} \sqrt{\frac{4N}{d}}\right) \geq (1-\theta)^2 / 6.
\]
With the choice \(\theta = 1/4\), it holds
\[
\PP\left(|Z_{k,j}| \geq \sqrt{\frac{N}{d}}\right) \geq \frac{9}{16} \times \frac{1}{6} = \frac{3}{32} \geq \frac{1}{11}.
\]
Finally, by the following lower bound,
\[
\frac{h_+(\mathcal V (x) )}{h_-(\mathcal V (x) )} = \exp \left( \max_{1\leq k < j \leq d } |Z_{k,j}| \right) \geq \exp (|Z_{1,2}|),
\]
we get, for any $N \geq d$,
\[
\PP\left(\frac{h_+(\mathcal V (x) )}{h_-(\mathcal V (x) )} \geq \exp\left(\sqrt{\frac{N}{d}}\right)\right) \geq \PP\left(\exp(|Z_{1,2}|) \geq \exp\left(\sqrt{\frac{N}{d}}\right)\right) = \PP\left(|Z_{1,2}| \geq \sqrt{\frac{N}{d}}\right) \geq \frac{1}{11}.\]
\qed

\subsection*{Proof of Proposition \ref{prop:diam_CRT}}

 As in the proof of Proposition \ref{prop:diamURT2}, notice that
\[
 \pr(  \diam (\mathcal{V}(x))\geq t) \leq d\pr \left(h_1 \geq \frac{t}{\sqrt{d}}\right).
\]
Then, for any $r\in (0,1)$ and $\lambda>0$,
\begin{align*}
    \pr(h_1 \geq r^N) & = \pr\left(\prod_{i=1}^N 2^{-B_i^{(1)}} \geq r^N\right)\\
    & \leq \mathbb E\left[\left(\frac{\prod_{i=1}^N 2^{-B_i^{(1)}}}{r^N}\right)^\lambda \right]=\left(\frac{\mathbb E\left[{2^{-\lambda B_1^{(1)}}}\right]}{r^{\lambda}}\right)^N. 
\end{align*}
It holds
\[
 \mathbb E\left[{2^{-\lambda B_1^{(1)}}}\right]= \frac{1}{d2^{\lambda}}+1-\frac{1}{d}.
\]
Hence,
\[
 \pr(h_1 \geq r^N) \leq \left(\frac{1}{d2^{\lambda}}+1-\frac{1}{d}\right)^N r^{-\lambda N}.
\]
Let us set $r=2^{-\alpha}$ and define 
\[
h(\lambda)= Q(\lambda) 2^{\lambda \alpha}
\]
with
\[
Q(\lambda)= \frac{1}{d2^{\lambda}}+1-\frac{1}{d}.
\]
By differentiating in $\lambda$, we get
\[
h^{\prime}(\lambda)=\log(2) 2^{\lambda \alpha}\left(\alpha Q(\lambda)-\frac{1}{d2^\lambda}\right).
\]
Hence, $h^{\prime}(\lambda_0)=0$ for $\lambda_0$ such that $2^{-\lambda_0}=\beta=\alpha (d-1)/(1-\alpha)$ and $\alpha \in (0,1/d)$. With this choice of $\lambda$, 
\[
 \PP(\diam(\mathcal{V}(x))\geq \sqrt{d}2^{-\alpha N}) \leq d\left(1-\frac{1-\beta}{d}\right)^N \beta^{-\alpha N}.
\]
We proceed in the same way as before for the diameter upper bound. By a union bound and symmetry in the directions, we have
\[
 \pr(  \diam (\mathcal{V}(x))\leq t) \leq d\pr \left(h_1 \leq \frac{t}{\sqrt{d}}\right).
\]
Then, for any $r\in (0,1)$ and $\lambda>0$,
\begin{align*}
    \pr(h_1 \leq r^N) & = \pr\left(\prod_{i=1}^N 2^{B_i^{(1)}} \geq r^{-N}\right)\\
    & \leq \mathbb E\left[\left(\frac{\prod_{i=1}^N 2^{B_i^{(1)}}}{r^{-N}}\right)^\lambda \right]=\left(\frac{\mathbb E\left[{2^{\lambda B_1^{(1)}}}\right]}{r^{-\lambda}}\right)^N. 
\end{align*}
It holds
\[
 \mathbb E\left[{2^{\lambda B_1^{(1)}}}\right]= \frac{2^{\lambda}}{d}+1-\frac{1}{d}.
\]
Hence,
\[
 \pr(h_1 \leq r^N) \leq \left(\frac{2^{\lambda}}{d}+1-\frac{1}{d}\right)^N r^{\lambda N}.
\]
Let us set $r=2^{-\alpha}$ and denote 
\[
h(\lambda)= Q(\lambda) 2^{-\lambda \alpha}
\]
with
\[
Q(\lambda)= \frac{2^{\lambda}}{d}+1-\frac{1}{d}.
\]
By differentiating in $\lambda$, we get
\[
h^{\prime}(\lambda)=\log(2) 2^{-\lambda \alpha}\left(\frac{2^\lambda}{d}-\alpha Q(\lambda)\right).
\]
Hence, $h^{\prime}(\lambda_0)=0$ for $\lambda_0$ such that $2^{\lambda_0}=\beta=\alpha (d-1)/(1-\alpha)$ and $\alpha \in (1/d , 1)$. With this choice of $\lambda$, 
\[
 \PP(\diam(\mathcal{V}(x))\leq \sqrt{d}2^{-\alpha N}) \leq d\left(1-\frac{1-\beta}{d}\right)^N \beta^{-\alpha N}.
\]
\qed

\subsection*{Proof of Corollary \ref{cor22}}

According to Proposition \ref{prop:diam_CRT}, for any $\alpha \in (1/d,1)$, we have for $\beta = \alpha (d-1)/(1-\alpha)$ the inequality
\begin{equation*}
    \PP(\diam(\mathcal{V}(x))\leq \sqrt{d}2^{-\alpha N}) \leq d\left(1-\frac{1-\beta}{d}\right)^N \beta^{-\alpha N}.
\end{equation*}
Take now $\alpha=\alpha_N=1/d+\gamma_N$, with $\gamma_N \rightarrow_{N\rightarrow +\infty} 0$. In ths case,
\[
\beta=\beta_N=\frac{\alpha_N(d-1)}{1-\alpha_N}=(d-1)\frac{1+d\gamma_N}{d-1-d\gamma_N}=1+a_d \gamma_N+ b_d \gamma_N^2+O(\gamma_N^3),
\]
where $a_d=d^2/(d-1)$ and $b_d=d^3/(d-1)^2$. This gives
\[
\log\left(1-\frac{1-\beta}{d}\right)= \frac{a_d}{d} \gamma_N+ \frac{b_d}{d}\gamma_N^2 - \frac{a^2_d}{2d^2}\gamma_N^2 + O(\gamma_N^3).
\]
Futhermore $$\log(\beta) = - a_d \gamma_N + b_d \gamma_N^2 -  \frac{a_d^2}{2} \gamma_N^2 + O(\gamma_N^3)$$
so $$\alpha \log(\beta) = \frac{a_d}{d} \gamma_N + \frac{b_d}{d} \gamma_N^2 -  \frac{a_d^2}{2d} \gamma_N^2 + a_d \gamma_N^2 + O(\gamma_N^3).$$
Then 
\begin{eqnarray*}
    \log\left(1-\frac{1-\beta}{d}\right) - \alpha \log(\beta) &=& -\frac{a^2_d}{2d^2}\gamma_N^2 + \frac{a_d^2}{2d} \gamma_N^2 - a_d \gamma_N^2 + O(\gamma_N^3) \\ &=& - a_d \gamma_N^2 \left( 1 + \frac{a_d}{2d^2} - \frac{a_d}{2d} \right) + O(\gamma_N^3).
\end{eqnarray*}
Moreover 
$$1 + \frac{a_d}{2d^2} - \frac{a_d}{2d}  = 1 + \frac{1}{2(d-1)} - \frac{d}{2(d-1)} = 1 - \frac{1}{2} = \frac{1}{2}.$$
Finally
\begin{eqnarray*}
  \left(1-\frac{1-\beta}{d}\right)^N \beta^{-\alpha N} &=& \exp \left( N \log\left(1-\frac{1-\beta}{d}\right) - N \alpha \log(\beta)    \right) \\ &=& \exp \left( - \frac{a_d}{2} N \gamma_N^2 + O(N \gamma_N^3)  \right).  
\end{eqnarray*}
Choosing $\gamma_N=2 \sqrt{\log(N)/(a_d N)} \in (0 , 1 - 1/d)$ for $N$ large enough, gives
$$\left(1-\frac{1-\beta}{d}\right)^N \beta^{-\alpha N} = \exp \left( - 2 \log(N) + O \left(\log(N)^{3/2} / \sqrt{N}   \right) \right) \underset{ {N \to + \infty}}\sim N^{-2}$$
and concludes the proof via the Borel-Cantelli lemma. Moreover for the upper bound of the diameter we use also Proposition \ref{prop:diam_CRT}. For any $\alpha \in (0,1/d)$, we have for $\beta = \alpha (d-1)/(1-\alpha)$,
\[
 \PP(\diam(\mathcal{V}(x))\geq \sqrt{d}2^{-\alpha N}) \leq d\left(1-\frac{1-\beta}{d}\right)^N \beta^{-\alpha N}.
\]
Let us take here $\alpha=\alpha_N=1/d-\gamma_N$, with $\gamma_N \rightarrow_{N\rightarrow +\infty} 0$. In ths case,
\[
\beta=\beta_N=\frac{\alpha_N(d-1)}{1-\alpha_N}=(d-1)\frac{1-d\gamma_N}{d-1+d\gamma_N}=1-a_d \gamma_N+ b_d \gamma_N^2+O(\gamma_N^3),
\]
where $a_d=d^2/(d-1)$ and $b_d=d^3/(d-1)^2$. This gives
\[
\log\left(1-\frac{1-\beta}{d}\right)= - \frac{a_d}{d} \gamma_N+ \frac{b_d}{d}\gamma_N^2 - \frac{a^2_d}{2d^2}\gamma_N^2 + O(\gamma_N^3).
\]
Futhermore $$\log(\beta) = - a_d \gamma_N + b_d \gamma_N^2 -  \frac{a_d^2}{2} \gamma_N^2 + O(\gamma_N^3)$$
so $$\alpha \log(\beta) = - \frac{a_d}{d} \gamma_N + \frac{b_d}{d} \gamma_N^2 -  \frac{a_d^2}{2d} \gamma_N^2 + a_d \gamma_N^2 + O(\gamma_N^3).$$
Then 
\begin{eqnarray*}
    \log\left(1-\frac{1-\beta}{d}\right) - \alpha \log(\beta) &=& -\frac{a^2_d}{2d^2}\gamma_N^2 + \frac{a_d^2}{2d} \gamma_N^2 - a_d \gamma_N^2 + O(\gamma_N^3) \\ &=& - a_d \gamma_N^2 \left( 1 + \frac{a_d}{2d^2} - \frac{a_d}{2d} \right) + O(\gamma_N^3).
\end{eqnarray*}
Moreover 
$$1 + \frac{a_d}{2d^2} - \frac{a_d}{2d}  = 1 + \frac{1}{2(d-1)} - \frac{d}{2(d-1)} = 1 - \frac{1}{2} = \frac{1}{2}.$$
Finally
\begin{eqnarray*}
  \left(1-\frac{1-\beta}{d}\right)^N \beta^{-\alpha N} &=& \exp \left( N \log\left(1-\frac{1-\beta}{d}\right) - N \alpha \log(\beta)    \right) \\ &=& \exp \left( - \frac{a_d}{2} N \gamma_N^2 + O(N \gamma_N^3)  \right).  
\end{eqnarray*}
Choosing $\gamma_N=2 \sqrt{\log(N)/(a_d N)}$ gives
$$\left(1-\frac{1-\beta}{d}\right)^N \beta^{-\alpha N} = \exp \left( - 2 \log(N) + O \left(\log(N)^{3/2} / \sqrt{N}   \right) \right) \underset{ {N \to + \infty}}\sim N^{-2}$$
and concludes the proof via the Borel-Cantelli lemma.

The last inequality follows directly by invoking the two previous inequalities on diameter and volume.
\qed

\subsection*{Proof of Corollary \ref{cor_UcR}}

The proof follows the same steps as the one of Corollary \ref{cor_URT}. First, since the local regression map results from a tree construction, each element in $\mathcal V ( x, (D_i,S_i)_{i=1}^N )$ is a rectangle. Hence, in light of \cite{wenocur1981some}, it holds that the resulting local map has dimension $v = 2d$.

Second, according to Theorem \ref{th2:general} (applied pointwise for $x\in S_X$), we find that whenever $ n \PP(\mathcal V ( x )) /\log(n) \to \infty$, it holds that 
$ \sum_{n\geq 1} \mathbb P ( |\hat g_{\mathcal V}(x) - g(x)| > v_n ) < \infty
$
where \[v_n = \sqrt{ 3 \sigma^2 \log( {(n+1)^{v+2}}  ) / (n  \kappa f_X(x) \lambda ( \mathcal V(x)) )}  + L(\mathcal V(x)) \diam (\mathcal V(x)).\] Note that we have set $\delta = (n+1)^{-2}$.  Making use of Borel Cantelli Lemma, it implies that with probability $1$, for $n$ large enough,
$$|\hat g_{\mathcal V}(x) - g(x)| \leq  \sqrt{\frac{ 3 \sigma^2 \log\left( {(n+1)^{v+2}} \right)}{n  \kappa f_X(x) \lambda ( \mathcal V(x)) }} + L(\mathcal V(x)) \diam (\mathcal V(x)) .$$

Third, using \ref{cond:density_X}, it follows that
$$ \mathbb P (\mathcal{V}(x) ) \geq \kappa f_X(x) \lambda(\mathcal{V}(x)) = \kappa f_X(x) 2^{-N} = n^{-d/(d+2)} \kappa f_X(x) $$
then $$ n \mathbb P (\mathcal{V}(x) ) \geq n^{2/(d+2)} \kappa f_X(x).$$
Hence, we get that $ n \mathbb P (\mathcal{V}(x) ) / \log(n) \to \infty$.

Fourth, by putting together the second and third point from above, we have the following inequality, with probability $1$, for $n $ large enough,
$$|\hat g_{\mathcal V}(x) - g(x)| \leq  \sqrt{\frac{ 3 \sigma^2 \log\left( (n+1)^{v+2} \right)}{n  \kappa f_X(x) \lambda ( \mathcal V(x)) }} + L(\mathcal V(x)) \diam (\mathcal V(x)).$$
%
%

Now, from Proposition \ref{cor22}, for a sufficiently large, we have
\begin{align*}
& \diam (\mathcal{V}(x)) \leq \sqrt d 2^{-N/d+2\sqrt{(d-1)N\log(N)/d^2}}, \\
& \lambda(\mathcal{V}(x)) = 2^{-N}.
\end{align*}

Hence, we get, with probability $1$, for $n$ large enough,
$$| \hat g_{\mathcal V}(x) - g(x)| \leq \sqrt{\frac{ 3 \sigma^2 \log\left( {(n+1)^{v+2}} \right)}{n \kappa f_X(x) 2^{-N}}} + L(\mathcal V(x)) \sqrt d 2^{-N/d+2\sqrt{(d-1)N\log(N)/d^2}}.$$

Because \( N = d\log(n)/(\log(2)(d+2)) \), we obtain
\begin{align*}
    | \hat g_{\mathcal V}(x) - g(x)| &\leq n^{-1/(d+2)} \sqrt{\frac{ 3 \sigma^2 \log\left( {(n+1)^{v+2}} \right)}{\kappa f_X(x)}} + n^{-1/(d+2)} L(\mathcal V(x)) \sqrt d e^{2\sqrt{(d-1)N\log(N)/d^2}}\\
& \leq n^{-1/(d+2)} \sqrt{\frac{ 6 \sigma^2 (v+2)\log\left( n \right)}{\kappa f_X(x)}} + n^{-1/(d+2)} L(\mathcal V(x)) \sqrt d e^{2\sqrt{(d-1)N\log(N)/d^2}}
\\
& \leq n^{-1/(d+2)} \sqrt{\frac{ 12 \sigma^2 (d+1)\log\left( n \right)}{\kappa f_X(x)}} + n^{-1/(d+2)} L(\mathcal V(x)) \sqrt d e^{\sqrt{N\log(N)}}
\end{align*}
where we use $v = 2d$ and the inequality $2 \sqrt{(d-1)/d^2} \leq 1$ since $(d-2)^2 \geq 0$.

For $n$ large enough, we have \( N = d\log(n)/(\log(2)(d+2)) \geq 8.\) Thus, this implies that $\log(n) = \frac{d+2}{d} \log(2) N \leq 3 \log(2) N = \log(8) N \leq \log(N) N$. 
Moreover $N \leq 2 \log(n)$, and for $n$ large enough $N \leq \log(n)^2$ then $\log(N) \leq 2 \log \log(n).$ Finally $\log(n) \leq \log(N) N \leq  4 \log(n) \log(\log(n)).$ We conclude by using the inequality $\sqrt{x} \leq e^{\sqrt{x}}$ for $x= \log(n)$ and setting $C = \sqrt{{12 \sigma^2 (d+1)\log\left( n \right)}/({\kappa f_X(x)})} +  L(\mathcal V(x)) \sqrt d$.
\qed

\subsection*{Proof of Proposition \ref{centered tree not regular}}

We follow the proof of Proposition \ref{uniform tree not regular}, with similar notation, but this time the variable $E_i := -\log(U_i)$ is replaced by $E_i := \log(2)$. By performing the calculations again, we find the moments with lemma \ref{lemme moment},
\begin{eqnarray*}
    \mathbb{E}(Z_{k,j}^2) &=& 2N \log(2)^2/d. \\
    \mathbb{E}(Z_{k,j}^4) &=& N \mathbb{E}({V_i^{k,j \, 4}}) \mathbb{E}(E_i^4) + 3N (N-1) \mathbb{E}({V_i^{k,j \, 2}})^2 \mathbb{E}(E_i^2)^2 \\ &=& N \times \dfrac{2}{d} \times  \log(2)^4 + 3N(N-1) \left( \dfrac{2}{d}\right)^2 \times \log(2)^4 \\ &=& \dfrac{2N}{d^2} \log(2)^4 (6N - 6 + d).
\end{eqnarray*}

The Paley-Zygmund bound becomes
\[
\frac{\mathbb{E}(Z_{k,j}^2)^2}{\mathbb{E}(Z_{k,j}^4)} = \frac{4 \log(2)^4 N^2}{d^2} \times \frac{d^2}{2N \log(2)^4 (6N - 6 + d) }= \frac{2N}{6N - 6 + d}.
\]

Thus, for all \(\theta \in (0,1)\) and $N \geq d$,
\[
\PP\left(|Z_{k,j}| \geq \sqrt{\theta} \sqrt{\frac{2N \log(2)^2}{d}}\right) \geq (1-\theta)^2 \frac{2N}{6N - 6 + d} \geq \dfrac{2(1-\theta)^2}{7}.
\]

Let us choose \(\theta = 1/2\) to obtain
\[
\PP\left(|Z_{k,j}| \geq \log(2) \sqrt{\frac{N}{d}}\right) \geq \dfrac{1}{14}
\]

and thus for $N \geq d$,
\[
\PP\left(\dfrac{h_+(\mathcal V (x))}{h_-(\mathcal V (x))} \geq 2^{\sqrt{{N}/{d}}} \right) \geq \frac{1}{14}.
\]
Thus, with probability at least $1/14$, the ratio ${h_+(\mathcal V (x))}/{h_-(\mathcal V (x))} $ is bounded below by a quantity that grows exponentially towards infinity. This means that uniform trees are not regular.
\qed

\subsection*{Proof of Proposition \ref{mondrian}}

\sloppypar{}    According to the paper by \cite{minimaxmondrian} (see proposition 1), we know the distribution of the largest and the smallest side. In fact, we have \( h_-(\mathcal V (x)) \sim \min(X_1, \dots, X_d) \) and \( h_+(\mathcal V (x)) \sim \max(X_1, \dots, X_d) \), where the \( X_i \) are i.i.d. and follow the Gamma distribution \( X \sim \Gamma(2, \lambda) \). We have for all $u \geq 0$
\[ \mathbb{P}(X \geq u) = \int_u^{+\infty} \lambda^2 t e^{-\lambda t} \, dt = e^{-\lambda u} (1 + u \lambda) \geq e^{-\lambda u}. \]
Moreover,
\[ \mathbb{P}(h_-(\mathcal V (x)) \geq u) = \mathbb{P}(X \geq u)^d \geq e^{-\lambda u d} = 1 - \delta \]
for \( u = -{\log(1-\delta)}/{(\lambda d)}. \)
Then with probability at least $1 - \delta$,
$$h_-(\mathcal V (x)) \geq - \dfrac{\log(1-\delta)}{\lambda d}.$$
We focus now on \( h_+(\mathcal V (x)) \). We have for all $t \geq 0$,
\[ \mathbb{P}(h_+(\mathcal V (x)) \leq t) = \mathbb{P}(X \leq t)^d. \]
Let \( Y := X - \EE(X) \). Since \( X \) follows a Gamma distribution, \( Y \) is Sub-Gamma. According to \cite{boucheron2013concentration} (page 29),
\[ \forall t > 0, \quad \mathbb{P}(\lambda Y \geq 2\sqrt{t} + t) \leq e^{-t}. \]
Thus,
\begin{eqnarray*}
    \mathbb{P}(\lambda h_+ \leq 2 \sqrt{t} + t + \lambda \EE(X)) &=& \mathbb{P}(\lambda Y \leq 2 \sqrt{t} + t)^d = \left(1 - \mathbb{P}(\lambda Y > 2 \sqrt{t} + t)\right)^d \\ &\geq& (1 - e^{-t})^d = 1 - \delta
\end{eqnarray*}
with \( t = - \log(1-(1-\delta)^{1/d}). \)
Therefore, with probability at least \(1 - \delta\),
\[ h_+(\mathcal V (x)) \leq \frac{2 + 2\sqrt{- \log(1-(1-\delta)^{1/d})} - \log(1-(1-\delta)^{1/d})}{\lambda}. \]
In particular, for \( \delta \leq 1 - (1 - e^{-1})^d \),
\[ h_+(\mathcal V (x)) \leq \frac{-5 \log(1-(1-\delta)^{1/d})}{\lambda} \leq \dfrac{-5\log(\delta/d)}{\lambda} \]
where the last inequality comes from the inequality \( \delta/d \leq 1 - (1 - \delta)^{1/d} \).
Hence, with probability at least \(1 - 2\delta\) for $\delta \leq 1 - (1 - e^{-1})^d$
\[ \frac{h_+(\mathcal V (x))}{h_-(\mathcal V (x))} \leq \frac{5d \log(\delta/d)}{{\log(1-\delta)}}. \]
\qed

\medskip
\section{Auxiliary results}

Let us state the following Vapnik-type inequality \cite{vapnik2015uniform}, which involves some standard-error normalization. The first inequality in the next theorem is Theorem 2.1 in  \cite{anthony1993result} (see also Theorem 1.11 in  \cite{lugosi2002pattern}). The second inequality can be obtained from the first one.

\begin{theorem}[normalized Vapnik inequality] \label{th:vapnik_normalized}
Let $(Z, Z_1, \ldots , Z_n) $ is a collection of random variables independent and identically distributed with common distribution $\mathbb P$ on $(S,\mathcal S)$. For any class $\mathcal A\subset \mathcal S $, $\delta > 0$ and $n\geq 1$, it holds with probability at least $1 -\delta $, for all $A\in \mathcal A$,
$$ \PP_n(A) \geq \PP(A) \left(1-\sqrt { \frac{4 \log( 4\mathbb S_\mathcal A(2n) / \delta) }{ n \PP (A)} } \right).  $$ 
In particular, with probability at least $1-\delta$ we have, for all $A\in \mathcal A$,
$$   \PP(A) \leq \dfrac{4}{n} \log\left(\dfrac{4\, \mathbb S_\mathcal A(2n)}{\delta} \right) + 2 \PP_n(A).$$ 
\end{theorem}

\begin{proof}
    The first statement is proved in \cite{anthony1993result}.   Let us  prove the second statement. According to the first point, with probability at least $1 - \delta$, we have for all $A \in \mathcal{A}$ $$n \PP_n(A) - n \PP(A) \geq     - \sqrt{  4n\PP(A)    \log(4\mathbb S_{\mathcal A}(2n) /\delta )   },$$   equivalently,  
    $$n\PP(A) - \sqrt{  4n\PP(A)    \log(4\mathbb S_{\mathcal A}(2n) /\delta )   } -  n \PP_n(A)\leq 0 .$$
    Setting $x = \sqrt{n\PP(A)}$, $\alpha = \sqrt{  4 \log(4\mathbb S_{\mathcal A}(2n) /\delta )}$ and $\beta = n\PP_n(A)$, we have that $x^2 - \alpha x - \beta \leq 0.$ Solving the inequality, we find  $$( \alpha - \sqrt{\alpha^2 + 4\beta}) / 2 \leq x \leq ( \alpha + \sqrt{\alpha^2 + 4\beta} ) / 2.$$ 
    By using the fact that \(x\) is positive and squaring both sides, it follows that $x^2 \leq ( \alpha + \sqrt{\alpha^2 + 4\beta})^2 / 4$. And by using the inequality \((a+b)^2 \leq 2(a^2 + b^2)\), we obtain $n\PP(A) = x^2 \leq \alpha^2 + 2\beta = 4 \log(4\mathbb S_{\mathcal A}(2n) /\delta ) + 2n\PP_n(A)$
    which is the desired result by dividing each side of the inequality by \(n\).   
\end{proof}

\vspace{5pt}

For more details, one can also refer to the book by \cite{boucheron2013concentration}, especially chapters 12 and 13, as well as \cite{devroye96probabilistic}. 



The following result is standard and known as the multiplicative Chernoff bound for empirical processes. The following version can be found in \cite{hagerup1990guided}.
\vspace{5pt}
\begin{theorem}\label{lemma=chernoff}
 Let $A$ be a set in $\mathbb R^d$.  For any $\delta \in (0,1)$ and all $n\geq 1$, we have with probability at least $1-\delta$
\begin{align*}
\PP_n(A) \geq \left(1- \sqrt{ \frac{2 \log(1/\delta)  }{ n\PP(A) } } \right) \PP(A)  .
\end{align*}
 In addition, for any $\delta \in (0,1)$ and $n\geq 1$, we have with probability at least $1-\delta$
\begin{align*}
\PP_n(A)  \leq \left(1 +  \sqrt{ \frac{3 \log(1/\delta)   }{ n \PP(A)} }  \right) \PP(A).
\end{align*}

\end{theorem}

\vspace{5pt}
This lemma is useful for proving Propositions \ref{uniform tree not regular} and \ref{centered tree not regular}.
\begin{lemma}\label{lemme moment}
   Let  \( M, (M_i)_{i=1,\ldots, N} \) be a collection of independent and identically distributed random variables such that $\mathbb E[ M^4] <\infty$ . It holds

\[
\mathbb{E}\left[ \left( \sum_{i=1}^{N} M_i \right)^4\right] = N \mathbb{E}(M^4) + 3N (N-1) \mathbb{E}(M^2)^2.
\]
\end{lemma}

\begin{proof}
We have
   \[
   \left( \sum_{i=1}^{N} M_i \right)^4 = \sum_{i,p,q,r=1}^{N} M_i M_p M_q M_r.
   \]
Since the \( M_i \) are independent and centered, the expectation of each product \( M_i M_p M_q M_r \) will be zero if at least one of the indices is distinct. This restricts the analysis to cases where all indices are identical or two pairs of indices are identical. If all indices are identical, i.e. \( i = p = q = r \), then the expectation of \( M_i^4 \) contributes to the sum: $ \sum_{i=1}^{N} \mathbb{E}(M_i^4) = N \, \mathbb{E}(M^4)$. When two indices are identical and the other two are also identical, i.e. \( i = p \neq q = r \), we get a product of the form \( M_i^2 M_q^2 \). We have 3 choices either $i$ is equal to \( p \), \( q \), or \( r \)). The remaining two indices must necessarily be equal. This yields: $ 3 \sum_{i \neq q} \mathbb{E}(M_i^2) \, \mathbb{E}(M_q^2) = 3N (N - 1) \, \mathbb{E}(M^2)^2.$ Combining the two terms, this proves the result.
\end{proof}

\newpage
\bibliography{b2}

\end{document}